\documentclass{amsart}
\usepackage[dvips,final]{graphics}
\usepackage{array}
\usepackage{arydshln}
\usepackage[makeroom]{cancel}
 \usepackage[all]{xy}
 \usepackage{url}
\usepackage{multirow, blkarray}
\usepackage{booktabs}
\usepackage{textcomp}
 \usepackage[final]{epsfig}
 \usepackage{color}
\usepackage[T1]{fontenc}      
\usepackage[english,french]{babel}
\usepackage[utf8]{inputenc}
\usepackage{blindtext}

\usepackage{amsfonts,amscd,array, mathdots, epigraph}
\usepackage{amsmath}
\usepackage{amssymb}
\usepackage{amsthm}
\usepackage{mathrsfs}
\usepackage{stmaryrd}
\usepackage{slashbox}
\usepackage{diagbox}
\usepackage{enumitem}

\usepackage{ulem}
\usepackage{tikz}
\usepackage{xcolor}
\usepackage{multicol}
\definecolor{ufogreen}{rgb}{0.24, 0.82, 0.44}

\begin{document}


\newtheorem{theorem}{Théorème}[section]
\newtheorem{theore}{Théorème}
\newtheorem{definition}[theorem]{Définition}
\newtheorem{proposition}[theorem]{Proposition}
\newtheorem{corollary}[theorem]{Corollaire}
\newtheorem*{con}{Conjecture}
\newtheorem*{remark}{Remarque}
\newtheorem*{remarks}{Remarques}
\newtheorem*{pro}{Problème}
\newtheorem*{examples}{Exemples}
\newtheorem*{example}{Exemple}
\newtheorem{lemma}[theorem]{Lemme}


\title{Éléments de comptage sur les générateurs du groupe modulaire et les $\lambda$-quiddités}

\author{Flavien Mabilat}

\date{}


\address{
}
\def\emailaddrname{{\itshape Courriel}}
\email{flavien.mabilat@univ-reims.fr}

\maketitle

\selectlanguage{french}

\begin{abstract}
L'objectif de cet article est de compter les $n$-uplets d'éléments $(a_{1},\ldots,a_{n})$ d'entiers strictement positifs solutions de l'équation $\begin{pmatrix}
   a_{n} & -1 \\[4pt]
    1    & 0 
   \end{pmatrix}
\begin{pmatrix}
   a_{n-1} & -1 \\[4pt]
    1    & 0 
   \end{pmatrix}
   \cdots
   \begin{pmatrix}
   a_{1} & -1 \\[4pt]
    1    & 0 
    \end{pmatrix}=\pm M$ lorsque $M$ est égal aux générateurs du groupe modulaire $S=\begin{pmatrix}
    0 & -1 \\[4pt]
    1 & 0 
   \end{pmatrix}$ et $T=\begin{pmatrix}
   1 & 1 \\[4pt]
   0 & 1 
    \end{pmatrix}$. Pour faire cela, on va s'intéresser aux $\lambda$-quiddités, qui sont les solutions de l'équation lorsque $M=Id$ (liées aux frises de Coxeter), dont la dernière composante est fixée.
\\
\end{abstract}

\selectlanguage{english}
\begin{abstract}
The aim of this article is to count the $n$-tuples of positive integers $(a_{1},\ldots,a_{n})$ solutions of the equation $\begin{pmatrix}
   a_{n} & -1 \\[4pt]
    1    & 0 
   \end{pmatrix}
\begin{pmatrix}
   a_{n-1} & -1 \\[4pt]
    1    & 0 
   \end{pmatrix}
   \cdots
   \begin{pmatrix}
   a_{1} & -1 \\[4pt]
    1    & 0 
    \end{pmatrix}=\pm M$ when $M$ is equal to the generators of the modular group $S=\begin{pmatrix}
   0 & -1 \\[4pt]
    1    & 0 
   \end{pmatrix}$ and $T=\begin{pmatrix}
   1 & 1 \\[4pt]
    0    & 1 
    \end{pmatrix}$. To count these elements, we will study the $\lambda$-quiddities, which are the solutions of the equation in the case $M=Id$ (related to Coxeter's friezes), whose last component is fixed.
\\
\end{abstract}

\selectlanguage{french}

\thispagestyle{empty}

\noindent {\bf Mots clés:} groupe modulaire; générateurs; $3d$-dissection; quiddité.   
\\
\begin{flushright}
\og \textit{L'esprit n'avance que s'il a la patience de tourner en rond, c'est-à-dire d'approfondir.} \fg
\\Emil Cioran, \textit{Le mauvais démiurge.}
\end{flushright}

\section{Introduction}

Le groupe modulaire \[SL_{2}(\mathbb{Z})=
\left\{
\begin{pmatrix}
a & b \\
c & d
   \end{pmatrix}
 \;\vert\;a,b,c,d \in \mathbb{Z},\;
 ad-bc=1
\right\}\] a été l'objet de très nombreux travaux et compte parmi les objets mathématiques les plus étudiés. Ceci s'explique par ses multiples et fructueuses applications ainsi que par les remarquables propriétés que lui et ses sous-groupes vérifient. Parmi celles-ci, l'une des plus utiles, et des plus anciennes, est l'existence de parties génératrices à seulement deux éléments. Il existe de nombreux couples de matrices qui peuvent jouer ce rôle, mais on se concentrera avant tout dans ce texte sur les deux éléments suivants :
\[T=\begin{pmatrix}
 1 & 1 \\[2pt]
    0    & 1 
   \end{pmatrix}, S=\begin{pmatrix}
   0 & -1 \\[2pt]
    1    & 0 
   \end{pmatrix}.
 \] 

Ce choix est particulièrement intéressant car il permet d'obtenir une expression remarquable des matrices du groupe modulaire. En effet, on peut montrer (voir par exemple l'introduction de \cite{M1}) que pour tout $A \in SL_{2}(\mathbb{Z})$, il existe un entier strictement positif $n$ et des entiers strictement positifs $a_{1},\ldots,a_{n}$ tels que : \[A=T^{a_{n}}ST^{a_{n-1}}S\cdots T^{a_{1}}S=\begin{pmatrix}
   a_{n} & -1 \\[4pt]
    1    & 0 
   \end{pmatrix}
\begin{pmatrix}
   a_{n-1} & -1 \\[4pt]
    1    & 0 
   \end{pmatrix}
   \cdots
   \begin{pmatrix}
   a_{1} & -1 \\[4pt]
    1    & 0 
    \end{pmatrix}:=M_{n}(a_{1},\ldots,a_{n}).\]
		
\noindent Toutefois, il convient de noter que cette façon d'exprimer les matrices du groupe modulaire n'est pas unique (même s'il est possible en rajoutant la condition $n$ minimal d'avoir une forme d'unicité, voir \cite{MO} section 6). En effet, on dispose, par exemple, des deux écritures de $-Id$ donnée ci-dessous : 
\[-Id=M_{3}(1,1,1)=M_{4}(1,2,1,2).\]

\medskip

Par ailleurs, les matrices $M_{n}(a_{1},\ldots,a_{n})$ apparaissent dans de nombreux autres domaines. Elles jouent notamment un rôle décisif dans la construction des frises de Coxeter et interviennent également dans un certain nombre de problèmes, tels que l'écriture matricielle des équations de Sturm-Liouville discrètes ou l'expression des réduites des fractions continues de Hirzebruch-Jung (voir par exemple l'introduction de \cite{O}). Dans ce dernier cas, on peut même considérer que les matrices $M_{n}(a_{1},\ldots,a_{n})$ offrent une généralisation judicieuse des fractions continues négatives.
\\
\\ \indent Ces nombreuses applications et l'absence d'unicité évoquée plus haut incitent naturellement à chercher les différentes écritures associées à une matrice, ou un ensemble de matrices, fixé. Plus précisément, on va s'intéresser ici pour une matrice $M$ donnée à l'équation ci-dessous :

\begin{equation}
\label{a}
\tag{$E_{M}$}
M_{n}(a_1,\ldots,a_n)=\pm M.
\end{equation} 

\smallskip

\noindent On dira, en particulier, qu'une solution de \eqref{a} est de taille $n$ si cette solution est un $n$-uplet d'éléments de $\mathbb{N}^{*}$. Lorsque $M=Id$, l'équation \eqref{a} est appelée équation de Conway-Coxeter et ses solutions, qui sont étroitement liées à la construction des frises de Coxeter (voir par exemple \cite{CH}), sont désignées par le terme $\lambda$-quiddité\footnote{En général, pour éviter les confusions, on précise l'ensemble sur lequel on travaille. Ici, on restera exclusivement sur les entiers strictement positifs. Aussi, afin de ne pas alourdir inutilement ce qui va suivre, on omettra systématiquement cette précision.}. Ces dernières ont été étudié par V. Ovsienko qui a fourni une construction récursive et une description combinatoire de ces solutions utilisant des découpages de polygones (voir \cite{O} et la section suivante), généralisant ainsi un théorème antérieur dû à Conway et Coxeter (voir \cite{CoCo} pour le résultat initial et \cite{MO} Théorème 3.3 pour un énoncé correspondant au point de vue adopté ici). Grâce à celle-ci, C. Conley et V. Ovsienko ont obtenu des formules permettant de compter l'ensemble des $\lambda$-quiddités de taille fixée (voir \cite{CO} et la section suivante). Notons qu'il existe également d'autres formules de comptage qui permettent de connaître le nombre de solutions de $(E_{Id})$ lorsque les $a_{i}$ appartiennent à certains anneaux commutatifs unitaires finis (voir notamment les trois preuves différentes des formules de comptage sur $\mathbb{Z}/p^{n}\mathbb{Z}$ présentes dans \cite{SZ,BC,M4}). Par ailleurs, de nombreux autres éléments sur ces solutions ont également été obtenues (voir par exemple \cite{C,CH,M2,M3}).
\\
\\ \indent Au vu de la définition des matrices $M_{n}(a_{1},\ldots,a_{n})$, il est naturel de chercher des résultats analogues pour les matrices $S$ et $T$ qui sont à l'origine de l'expression particulière des éléments du groupe modulaire sur laquelle on se penche dans ce texte. On dispose déjà d'une construction récursive des solutions de $(E_{S})$ et de $(E_{T})$ et d'une description combinatoire de celles-ci (voir \cite{M1} et la section \ref{des}). L'objectif de cet article est de compter ces solutions, et d'utiliser les méthodes employées pour obtenir le nombre de solutions de \eqref{a} pour d'autres matrices $M$ utilisées parfois comme générateur de $SL_{2}(\mathbb{Z})$. Plus précisément, dans ce qui va suivre, on va reprendre les formules de C. Conley et V. Ovsienko et les méthodes utilisées pour les obtenir afin de pouvoir compter les $\lambda$-quiddités dont la dernière composante est fixée. Ces nouveaux éléments de comptage nous permettront ensuite d'obtenir le nombre de solutions de taille fixée de $(E_{M})$ pour $M=S$ ou $M=T$, ainsi que pour d'autres matrices qui leur sont immédiatement liées telles que $T^{-1}$ ou $TS$.
\\
\\ \indent Dans ce qui va suivre, $n$ est un entier positif et, pour tout réel $x$, $E[x]$ est la partie entière de $x$. De plus, on utilisera les conventions suivantes : ${n \choose k}={n \choose -1}=0$ si $k>n$ et ${-n \choose k}=1$. $Q_{0}:=1$ et, pour $n \geq 1$, $Q_{n}$ désigne le nombre de $\lambda$-quiddités de taille $n+2$. On note $Q:=\sum_{n=0}^{+\infty} Q_{n} X^{n}$ la série génératrice associée. On pose $P_{0}:=1$ et, pour $n \geq 1$, $P_{n}:=\sum_{k=0}^{E\left[n/3\right]} \frac{1}{n-k+1} {n-2k-1 \choose k}{2n-4k \choose n-3k}$. On note $P:=\sum_{n=0}^{+\infty} P_{n} X^{n}$ et $P^{-1}:=\sum_{n=0}^{+\infty} \tilde{P_{n}} X^{n}$ l'inverse multiplicatif de $P$ (la signification combinatoire de $P_{n}$ et de $\tilde{P_{n}}$ sera donnée dans la section suivante). Soient $n \geq 1$ et $k \geq 1$. $V_{k,n}$ représente le nombre de $\lambda$-quiddités de taille $n+2$ dont la dernière composante est égale à $k$. Par ailleurs, on pose $V_{k,0}:=0$ et $V_{k}:=\sum_{n=0}^{+\infty} V_{k,n} X^{n}$. Si $R(X)$ est une série génératrice, on note, pour tout $j \geq 0$, $[X^{j}]R(X)$ le coefficient de $X^{j}$ dans $R(X)$.
\\
\\On va démontrer dans la partie \ref{DCQ} les deux résultats ci-dessous :

\begin{theorem}
\label{10}

Soit $n \geq 1$. On a 
\[\tilde{P_{n}}=-\sum_{j=0}^{n-1} \sum_{k=0}^{E\left[\frac{n-j-1}{2}\right]} (-1)^{j-k} {j \choose k} \sum_{l=0}^{E\left[\frac{n-j-2k-1}{3}\right]} \frac{2j+2}{n+j-2k-l+1} {n-j-2k-2l-2 \choose l}{2n-4k-4l-1 \choose n-j-2k-3l-1}.\]

\end{theorem}

\begin{theorem}
\label{11}

Soit $k \geq 1$. On dispose des deux formules ci-dessous :
\[V_{k}=[(Q-1)P^{-1}](1-P^{-1})^{k-1},\]
\[V_{1}=(Q-1)P^{-1}=\sum_{n=1}^{+\infty} \left(\sum_{j=0}^{E\left[\frac{n-1}{3}\right]} \sum_{k=0}^{E\left[\frac{n-3j-1}{3}\right]} \frac{3j+1}{n-k} {n-3j-1-2k-1 \choose k} {2n-4k-3j-2 \choose n-3k-3j-1}\right)X^{n}.\]

\end{theorem}

En utilisant les formules permettant de calculer les coefficients de $P^{-1}$ et les formules de $Q_{n}$, on pourra obtenir les valeurs de $V_{k,n}$. Une fois celle-ci connue, on pourra connaître le nombre de solutions de taille fixée de $(E_{S})$ et de $(E_{T})$ grâce au théorème ci-dessous. Avant de l'énoncer, on pose, pour $n \geq 0$, $\mathcal{S}_{n}$ le nombre de solutions de $(E_{S})$ de taille $n+2$ et $\mathcal{T}_{n}$ le nombre de solutions de $(E_{T})$ de taille $n+2$.

\begin{theorem}
\label{12}

i) $\mathcal{S}_{0}=\mathcal{S}_{1}=\mathcal{S}_{2}=0$ et pour tout $n \geq 3$ on a :
\[\mathcal{S}_{n}=\sum_{d=2}^{n-1} (d-1) V_{d,n-1}.\]
\noindent ii) $\mathcal{T}_{0}=0$ et pour tout $n \geq 1$ on a :
\[\mathcal{T}_{n}=\mathcal{S}_{n-1}+Q_{n}.\]

\end{theorem}

Ce résultat sera prouvé dans la section \ref{SetT}. Ensuite, en réutilisant les méthodes employées pour démontrer ce théorème, on comptera, dans la sous-partie \ref{autre}, le nombre de solutions de l'équation $(E_{M})$ lorsque la matrice $M$ appartient à l'ensemble $\{T^{-1},ST,TS,STST,TSTS\}$.

\section{Description et comptage des $\lambda$-quiddités}
\label{DCQ} 

L'objectif de cette section est de rappeler les différents éléments connus sur les $\lambda$-quiddités qui nous seront utiles pour démontrer, dans la dernière sous-section, le théorème \ref{11}.

\subsection{$\lambda$-quiddités et $3d$-dissections}
\label{dissec}

Les solutions de $(E_{Id})$ possèdent certaines propriétés intéressantes. Elles sont notamment invariantes par permutations circulaires et par retournement (c'est-à-dire que $(a_{1},\ldots,a_{n})$ est solution de $(E_{Id})$ si et seulement si $(a_{n},\ldots,a_{1})$ l'est aussi). On connaît également l'ensemble des $\lambda$-quiddités de taille $n$ pour les petites valeurs de $n$ :
\begin{itemize}
\item $(E_{Id})$ n'a pas de solution de taille 2;
\item $(1,1,1)$ est la seule $\lambda$-quiddité de taille 3;
\item il n'y a que deux solutions de $(E_{Id})$ de taille 4, $(1,2,1,2)$ et $(2,1,2,1)$.
\end{itemize}
\noindent De plus, les $\lambda$-quiddités peuvent être représentées grâce à des découpages de polygones. Pour cela, on a besoin de la notion ci-dessous :

\begin{definition}
\label{21}

i) Une 3$d$-dissection\footnote{Pour des raisons inconnues, C. Conley et V. Ovsienko ont remplacé dans leurs travaux l'appellation \og $3d$-dissection \fg, introduite par V. Ovsienko, par le nom \og dissection $3$-périodique \fg. Ici, on a choisi de conserver la terminologie d'origine.} est un découpage d'un polygone convexe $\mathcal{P}$ par des diagonales ne se croisant qu'aux sommets de $\mathcal{P}$ et tel que chaque sous-polygone résultant de ce découpage possède un nombre de sommets égal à un multiple de $3$.
\\
\\ii) On choisit un sommet que l'on numérote 1 et un sens de rotation et on numérote les autres sommets en suivant ce dernier. On appelle quiddité associée à la 3$d$-dissection du polygone à $n$ sommets $\mathcal{P}$ la séquence $(a_{1},\ldots,a_{n})$ où $a_{i}$ est égal au nombre de sous-polygones utilisant le sommet $i$.

\end{definition}

Celles-ci dont reliées à notre problème via le résultat ci-dessous, qui généralise le théorème de Conway-Coxeter (voir \cite{CoCo}) :

\begin{theorem}[\cite{O}, Théorème 1]
\label{21bis}

Soit $n \geq 3$. Un $n$-uplet d'entiers strictement positifs solution de $(E_{Id})$ est une quiddité d'une 3$d$-dissection d'un polygone convexe à $n$ sommets et réciproquement.

\end{theorem}

\begin{remark}
{\rm On retrouve dans ce théorème l'invariance par permutations circulaires et par retournement des $\lambda$-quiddités.
}
\end{remark}

\begin{examples}

{\rm On donne ci-dessous quelques exemples de 3$d$-dissections avec leur quiddité.}

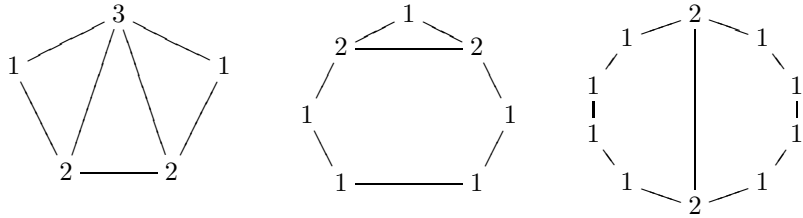
$$
\shorthandoff{; :!?}
\xymatrix @!0 @R=0.70cm @C=0.7cm
{
&&3\ar@{-}[rrd]\ar@{-}[lld]\ar@{-}[lddd]\ar@{-}[rddd]&
\\
1\ar@{-}[rdd]&&&& 1\ar@{-}[ldd]\\
\\
&2\ar@{-}[rr]&& 2
}
\qquad
\xymatrix @!0 @R=0.45cm @C=0.45cm
{
&&&1\ar@{-}[rrd]\ar@{-}[lld]
\\
&2\ar@{-}[ldd]\ar@{-}[rrrr]&&&& 2\ar@{-}[rdd]&\\
\\
1\ar@{-}[rdd]&&&&&& 1\ar@{-}[ldd]\\
\\
&1\ar@{-}[rrrr]&&&&1
}
\qquad
\xymatrix @!0 @R=0.32cm @C=0.45cm
 {
&&&2\ar@{-}[dddddddd]\ar@{-}[lld]\ar@{-}[rrd]&
\\
&1\ar@{-}[ldd]&&&& 1\ar@{-}[rdd]\\
\\
1\ar@{-}[dd]&&&&&&1\ar@{-}[dd]\\
\\
1\ar@{-}[rdd]&&&&&&1\ar@{-}[ldd]\\
\\
&1\ar@{-}[rrd]&&&& 1\ar@{-}[lld]\\
&&&2&
}
$$

\end{examples}

On peut compter, avec des méthodes classiques de dénombrement, le nombre $D_{n}$ de $3d$-dissections d'un polygone convexe à $n+2$ sommets. Dans l'article \cite{CO}, on trouve notamment la formule suivante : $D_{n}=\frac{1}{n+1}\sum_{k=0}^{E\left[\frac{n}{3}\right]} {n-2k-1 \choose k}{2n-3k \choose n-3k}$ (corollaire 3.6). Cependant, cette dernière ne nous est d'aucune aide pour compter le nombre de $\lambda$-quiddités de taille fixée (même si on a l'inégalité $Q_{n} \leq D_{n}$). En effet, il n'y a pas de bijection entre les $3d$-dissections et les quiddités des $3d$-dissections, comme l'illustre l'exemple ci-dessous :

$$
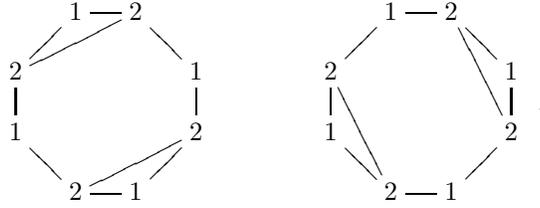
\shorthandoff{; :!?}
\xymatrix @!0 @R=0.40cm @C=0.4cm
{
&&1\ar@{-}[lldd] \ar@{-}[rr]&&2\ar@{-}[rrdd]\ar@{-}[lllldd]&
\\
&&&
\\
2\ar@{-}[dd]&&&&&& 1 \ar@{-}[dd]
\\
&&&&
\\
1&&&&&& 2\ar@{-}[lllldd]
\\
&&&
\\
&&2 \ar@{-}[rr]\ar@{-}[lluu] &&1 \ar@{-}[rruu]
}
\qquad
\qquad
\xymatrix @!0 @R=0.40cm @C=0.4cm
{
&&1\ar@{-}[lldd] \ar@{-}[rr]&&2\ar@{-}[rrdd]\ar@{-}[rrdddd]&
\\
&&&
\\
2\ar@{-}[dd]&&&&&& 1 \ar@{-}[dd]
\\
&&&&&&&.
\\
1&&&&&& 2
\\
&&&
\\
&&2 \ar@{-}[rr]\ar@{-}[lluuuu]\ar@{-}[lluu] &&1 \ar@{-}[rruu]
}
$$
\\
\\ \indent Par conséquent, pour effectuer le comptage souhaité, on a besoin de raffiner le concept introduit. Plus précisément, on utilise la relation d'équivalence suivante sur l'ensemble des $3d$-dissections : deux $3d$-dissections sont dites équivalentes si elles ont la même quiddité. Puis, on cherche un représentant particulier de chaque classe qui nous permettra d'utiliser des arguments combinatoires. Notons toutefois qu'il y a une bijection entre les quiddités des triangulations et les triangulations et que le nombre de ces quiddités est donné par les nombres de Catalan (voir \cite{CoCo}). Plus précisément, il y a $C_{n}=\frac{1}{n+1} {2n \choose n}$ quiddités de triangulations de taille $n+2$.

\subsection{Comptage des quiddités des $3d$-dissections}
\label{compt}

Tous les éléments développés dans cette sous-partie proviennent des différentes sections de \cite{CO}.
\\
\\ \indent Soient $n \geq 0$ et $\mathcal{P}$ un polygone convexe à $n+2$ sommets. On choisit un sommet que l'on numérote avec l'indice 1. Puis, on numérote les autres sommets en suivant le sens trigonométrique. On effectue ensuite une $3d$-dissection de $\mathcal{P}$. Celle-ci est constituée d'un certain nombre de sous-polygones dont les côtés peuvent être soit des côtés de $\mathcal{P}$, qu'on appellera alors côtés extérieurs, soit des diagonales, qu'on appellera côtés intérieurs. Les sous-polygones ne contiennent pas de diagonales intérieures. Dans la $3d$-dissection, il y a un seul sous-polygone utilisant le côté dont les sommets sont $1$ et $n+2$. Ce côté est appelé le côté de base de la $3d$-dissection et le sous-polygone utilisant ce côté, noté $\mathcal{P}_{b}$, est désigné comme étant le sous-polygone de base de la décomposition de $\mathcal{P}$. Considérons, s'il existe, un autre sous-polygone de la décomposition, noté $\mathcal{P'}$. Ce dernier possède un unique côté, nécessairement intérieur, tel que si on coupe $\mathcal{P}$ suivant ce côté on obtient un polygone contenant $\mathcal{P'}$ et un polygone contenant $\mathcal{P}_{b}$.

\begin{example}

{\rm Dans le dodécagone ci-dessous, la $3d$-dissection contient sept sous-polygones et le côté $2-9$ est le côté de base de l'hexagone $2-3-4-6-7-9$.}
$$
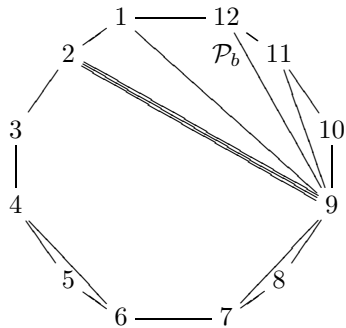
\shorthandoff{; :!?}
\xymatrix @!0 @R=0.5cm @C=0.7cm
 {
&&1\ar@{-}[rrrrddddd]\ar@{-}[ld]\ar@{-}[rr]&&12 \ar@{-}[rd]\ar@{-}[rrddddd]&
\\
&2\ar@{-}[ldd]\ar@{=}[rrrrrdddd]\ar@{-}[rrrrrdddd]&&&\mathcal{P}_{b}& 11\ar@{-}[rdd]\ar@{-}[rdddd]\\
\\
3\ar@{-}[dd]&&&&&&10\ar@{-}[dd]\\
\\
4\ar@{-}[rdd]\ar@{-}[rrddd]&&&&&&9\ar@{-}[ldd]\\
\\
&5\ar@{-}[rd]\ar@{-}&&&&8 \ar@{-}[ld]\\
&&6\ar@{-}[rr]&&7\ar@{-}[rruuu]&
}
\qquad
$$

\end{example}

Ce vocabulaire étant fixé, on peut maintenant définir les éléments nécessaires à l'obtention d'un représentant utile des classes d'équivalence.

\begin{definition}
\label{22}

Soient $n \geq 0$ et $\mathcal{P}$ un polygone convexe à $n+2$ sommets pour lequel on choisit une $3d$-dissection.
\\
\\i) On attribue au côté de base de la décomposition l'indice $\overline{0} \in \mathbb{Z}/3\mathbb{Z}$. Puis, on parcourt les côtés du sous-polygone de base dans le sens trigonométrique. Pour chaque côté on attribue l'indice du côté précédent augmenté de $\overline{1}$. Ensuite, on effectue la même chose pour les tous sous-polygones dont l'un des côtés a été indexé, en partant de ce côté. En continuant ce processus, on fournit un indice à tous les côtés extérieurs et intérieurs de la $3d$-dissection.
\\
\\ii) On dira que la $3d$-dissection est ouverte maximale si, pour tout sous-polygone $\mathcal{P}' \neq \mathcal{P}_{b}$ de la décomposition, tous les côtés de $\mathcal{P}'$, différents de son côté de base, dont l'indice est le même que celui du côté de base de ce sous-polygone sont des côtés extérieurs.
\\
\\iii) Si la propriété précédente est également vraie pour le sous-polygone de base alors on dira que la $3d$-dissection est ouverte maximale à base ouverte.

\end{definition}

\begin{remark}
{\rm La propriété demandée pour les sous-polygones dans ii) est toujours vérifiée pour les triangles.
}
\end{remark}

\begin{examples}

{\rm On donne ci-dessous deux $3d$-dissections d'un hexadécagone où l'indice est indiqué pour chacun des côtés. La première est ouverte maximale à base ouverte et la deuxième est ouverte maximale mais pas à base ouverte.}

$$
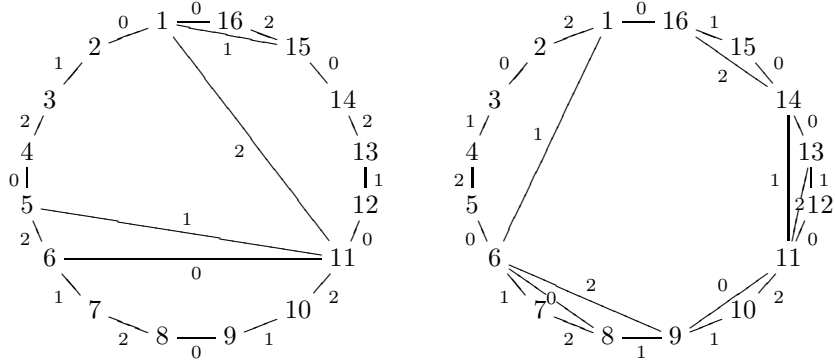
\shorthandoff{; :!?}
\xymatrix @!0 @R=0.35cm @C=0.30cm
 {
&&&&&&1\ar@{-}[llld]_{0}\ar@{-}[rrr]^{0}\ar@{-}[rrrrrrd]_{1}
\ar@{-}[rrrrrrrrddddddddd]_{2}&&&16\ar@{-}[rrrd]^{2}
\\
&&&2\ar@{-}[lldd]_{1}&&&&&&&&& 15\ar@{-}[rrdd]^{0}\\
\\
&3\ar@{-}[ldd]_{2}&&&&&&&&&&&&&14\ar@{-}[rdd]^{2}\\
\\
4\ar@{-}[dd]_{0}&&&&&&&&&&&&&&&13\ar@{-}[dd]^{1}\\
&&&&&&&&&&&&\\
5\ar@{-}[rdd]_{2}\ar@{-}[rrrrrrrrrrrrrrdd]^{1}&&&&&&&&&&&&&&&12\ar@{-}[ldd]^{0}\\
\\
&6\ar@{-}[rrdd]_{1}\ar@{-}[rrrrrrrrrrrrr]_{0}&&&&&&&&&&&&&11\ar@{-}[lldd]^{2}\\
\\
&&&7\ar@{-}[rrrd]_{2}&&&&&&&&& 10\ar@{-}[llld]^{1}\\
&&&&&&8\ar@{-}[rrr]_{0}&&&9
}
\qquad
\xymatrix @!0 @R=0.35cm @C=0.30cm
{
&&&&&&1\ar@{-}[llld]_{2}\ar@{-}[lllllddddddddd]_{1}\ar@{-}[rrr]^{0}&&&16\ar@{-}[rrrd]^{1}\ar@{-}[rrrrrddd]_{2}
\\
&&&2\ar@{-}[lldd]_{0}&&&&&&&&& 15\ar@{-}[rrdd]^{0}\\
\\
&3\ar@{-}[ldd]_{1}&&&&&&&&&&&&&14\ar@{-}[rdd]^{0}\\
\\
4\ar@{-}[dd]_{2}&&&&&&&&&&&&&&&13\ar@{-}[dd]^{1}\\
&&&&&&&&&&&&\\
5\ar@{-}[rdd]_{0}&&&&&&&&&&&&&&&~~12\ar@{-}[ldd]^{0}\\
\\
&6\ar@{-}[rrdd]_{1}\ar@{-}[rrrrrddd]|-{0}\ar@{-}[rrrrrrrrddd]^{2}&&&&&&&&&&&&&11\ar@{-}[lldd]^{2}\ar@{-}[ruuuu]|-{2}\ar@{-}[uuuuuu]^{1}\\
\\
&&&7\ar@{-}[rrrd]_{2}&&&&&&&&& 10\ar@{-}[llld]^{1}\\
&&&&&&8\ar@{-}[rrr]_{1}&&&9\ar@{-}[rrrrruuu]^{0}
}
$$

\end{examples}

Cette catégorie particulière de $3d$-dissections est extrêmement utile car elle offre un représentant très intéressant des classes d'équivalence, comme l'illustre les deux résultats ci-dessous :

\begin{theorem}[\cite{CO}, Théorème 6.10]
\label{23}

Chaque classe d'équivalence de $3d$-dissections contient une unique $3d$-dissection ouverte maximale.

\end{theorem}

\begin{remark}
{\rm Ce théorème nous permet de retrouver l'existence d'une bijection entre les triangulations et les quiddités de triangulations.
}
\end{remark}

\begin{proposition}
\label{24}

Soient $n \geq 0$ et $\mathcal{P}$ un polygone convexe à $n+2$ sommets pour lequel on choisit une $3d$-dissection que l'on regarde comme une collection de $3d$-dissections (éventuellement vides) accolées aux côtés du sous-polygone de base.
\\
\\i) La $3d$-dissection de $\mathcal{P}$ est ouverte maximale si et seulement si les $3d$-dissections adjacentes aux côtés du sous-polygone de base sont ouvertes maximales à base ouverte.
\\
\\ii) Supposons que la $3d$-dissection de $\mathcal{P}$ est ouverte maximale. Celle-ci est à base ouverte si et seulement il n'y a pas de $3d$-dissections accolées aux côtés d'indice 0 du sous-polygone de base.

\end{proposition}

Grâce à ces résultats, on peut compter, avec des méthodes de séries génératrices multivariées, le nombre $Q_{n}$ de $\lambda$-quiddités et le nombre $P_{n}$ de quiddités de $3d$-dissections ouvertes maximales à base ouverte d'un polygone convexe à $n+2$ sommets (ce décalage de deux entre les indices de $Q_{n}$ et de $P_{n}$ et le nombre de sommet du polygone a été introduit par analogie avec les nombres de Catalan). 

\begin{theorem}[\cite{CO}, Théorèmes 2.1 et 2.2]
\label{25}

Soit $n>0$. On a :
\[P(X)=1+X~P(X)^{2}+X^{4}P(X)^{4}+X^{7}P(X)^{6}+\ldots=1+\frac{X~P(X)^{2}}{1-X^{3}P(X)^{2}},\]
\[P_{n}=\sum_{k=0}^{E\left[\frac{n}{3}\right]} \frac{1}{n-k+1} {n-2k-1 \choose k}{2n-4k \choose n-3k},\]
\[Q(X)=1+X~P(X)^{2}+X^{4}P(X)^{5}+X^{7}P(X)^{8}+\ldots=1+\frac{X~P(X)^{2}}{1-X^{3}P(X)^{3}},\]
\[Q_{n}=\sum_{k=0}^{E\left[\frac{n}{3}\right]} \sum_{s=0}^{k} \frac{3(k-s)+2}{n-s+1} {n-3k+s-2 \choose s}{2n-3k-s-1 \choose n-3k-1}.\]

\end{theorem}



\subsection{Calcul des coefficients de $P^{-1}$}
\label{calc}

Pour obtenir les valeurs de $V_{k,n}$, on aura besoin des pouvoir calculer les coefficients de l'inverse de $P$ pour la multiplication de $\mathbb{Z}[[X]]$ (notons que ce dernier existe puisque $P_{0}=1$). Une des premières façons de procéder est d'utiliser la formule récursive générale donnant les coefficients de l'inverse multiplicatif d'une série formelle. En effet, on sait que si $\left(\sum_{n=0}^{+\infty} a_{n} X^{n}\right)\left(\sum_{n=0}^{+\infty} b_{n} X^{n}\right)=1$ dans $A[[X]]$ (avec $A$ un anneau commutatif unitaire) alors $b_{0}=a_{0}^{-1}$ et, pour tout $n \geq 1$, $b_{n}=-a_{0}^{-1}\sum_{k=1}^{n} a_{k}b_{n-k}$. On pourrait également utiliser des formules directes générales, telle que la formule de Wronski (voir \cite{H} Théorème 1.3) ou des formules utilisant des multi-indices (voir par exemple \cite{L} Théorème 11.13). Cependant, l'existence d'une équation fonctionnelle pour $P$ va nous permettre d'obtenir une formule plus intéressante, donnée dans le théorème \ref{10}. Avant de prouver celle-ci, on énonce un résultat intermédiaire dont on aura besoin.

\begin{proposition}[\cite{CO} lemme 4.3]
\label{26}

Soit $e \geq 1$.
\[[P(X)]^{e}=\sum_{n=0}^{+\infty} \left(\sum_{k=0}^{E\left[\frac{n}{3}\right]} \frac{e}{n-k+e} {n-2k-1 \choose k}{2n-4k+e-1 \choose n-3k}\right) X^{n}.\]

\end{proposition}

\begin{proof}[Démonstration du théorème \ref{10}]

D'après le le théorème \ref{25}, $P(X)=1+\frac{X~P(X)^{2}}{1-X^{3}P(X)^{2}}$. Ainsi,
\[P^{-1}(X)=\frac{1-X^{3}P(X)^{2}}{1-X^{3}P(X)^{2}+X~P(X)^{2}}=1-\frac{X~P(X)^{2}}{1-(X^{3}P(X)^{2}-X~P(X)^{2})}.\]
Par conséquent, on a :
\begin{eqnarray*}
P^{-1}(X) &=& 1-X~P(X)^{2} \sum_{j=0}^{+\infty} (X^{3}P(X)^{2}-X~P(X)^{2})^{j} \\
          &=& 1-X~P(X)^{2} \sum_{j=0}^{+\infty} P(X)^{2j}(X^{3}-X)^{j} \\
          &=& 1-\sum_{j=0}^{+\infty} X~P(X)^{2j+2} \sum_{k=0}^{j} {j \choose k} (-1)^{j-k}X^{3k}X^{j-k} \\
          &=& 1-\sum_{j=0}^{+\infty} \left(\sum_{k=0}^{j} {j \choose k} (-1)^{j-k}X^{j+2k+1}\right) P(X)^{2j+2}. \\
\end{eqnarray*}

\noindent On a $\tilde{P_{0}}=1$. Soit $n \geq 1$. Comme $j+2k+1>n$ si $j \geq n$ et $j+2k+1>n$ si $k>\frac{n-j-1}{2}$, on a, d'après la formule précédente :
\begin{eqnarray*}
\tilde{P_{n}} &=& -\sum_{j=0}^{n-1}\sum_{k=0}^{E\left[\frac{n-j-1}{2}\right]} (-1)^{j-k} {j \choose k} [X^{n-j-2k-1}]P(X)^{2j+2} \\
              &=& -\sum_{j=0}^{n-1}\sum_{k=0}^{E\left[\frac{n-j-1}{2}\right]} (-1)^{j-k} {j \choose k} \sum_{l=0}^{E\left[\frac{n-j-2k-1}{3}\right]} \frac{2j+2}{n+j-2k-l+1} {n-j-2k-2l-2 \choose l}{2n-4k-4l-1 \choose n-j-2k-3l-1} \\
							& & {\rm (proposition}~\ref{26}).
\end{eqnarray*}

\end{proof}

\noindent La signification combinatoire de $-\tilde{P_{n}}$ est donnée dans la prochaine sous-partie.

\subsection{Nombre de quiddités dont la dernière composante est fixée}
\label{preuve}

L'objectif de cette sous-section est d'obtenir des formules permettant de calculer $V_{k,n}$. Avant de rechercher ces dernières, notons que certaines valeurs de $V_{k,n}$ sont déjà connues. En effet, une $3d$-dissection d'un polygone convexe à $n+2$ sommets contient au plus $n$ sous-polygones. Si de plus un de ces sous-polygones n'est pas un triangle alors la décomposition contient au plus $n-3$ sous-polygones. Ainsi, si $k>n$ on a nécessairement $V_{k,n}=0$. De plus, si $n-2 \leq k \leq n$ alors les $3d$-dissections dont les quiddités ont $k$ comme dernière composante ne peuvent contenir que des triangles. Or, il y a une bijection entre les triangulations et les quiddités de triangulations. Donc, si on connaît le nombre de triangulations pour lesquels $k$ triangles, avec $k \in \{n-2,n-1,n\}$, utilisent le dernier sommet, on aura les formules pour $V_{n,n}$, $V_{n-1,n}$ et $V_{n-2,n}$. Grâce à ce raisonnement, on peut obtenir :

\begin{proposition}
\label{26bis}

Soit $n \geq 3$. $V_{n,n}=1$, $V_{n-1,n}=n-1$, $V_{n-2,n}=\frac{(n+1)(n-2)}{2}$. Par ailleurs, si $1 \leq k \leq n$, $V_{k,n}>0$.

\end{proposition}

\begin{proof}

i) Il n'y a qu'une seule triangulation pour laquelle $n$ triangles utilisent le dernier sommet, celle où les triangles ont pour sommets $i,i+1,n+2$ avec $1 \leq i \leq n$. Ainsi, $V_{n,n}=1$.
$$
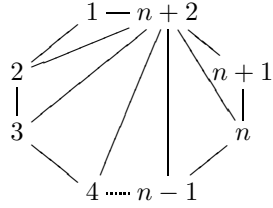
\shorthandoff{; :!?}
\xymatrix @!0 @R=0.40cm @C=0.5cm
{
&&1\ar@{-}[lldd] \ar@{-}[rr]&&n+2\ar@{-}[rrdd]\ar@{-}[dddddd]\ar@{-}[rrdddd]&
\\
&&&
\\
2\ar@{-}[dd]\ar@{-}[rrrruu]&&&&&& n+1 \ar@{-}[dd]
\\
&&&&
\\
3\ar@{-}[rrrruuuu]&&&&&& n
\\
&&&
\\
&&4\ar@{.}[rr]\ar@{-}[rruuuuuu] \ar@{-}[lluu] &&n-1 \ar@{-}[rruu]
}
$$
ii) Si $k=n-1$, il y a $n-1$ triangulations pour lesquels le dernier sommet est utilisé par $n-1$ sommets. En effet, pour construire une telle triangulation, il suffit de choisir un triangle $T$ qui n'utilise pas le sommet $n+2$, les $n-1$ triangles restants étant alors de la forme $i,i+1,n+2$. Or, ce triangle ne peut-être qu'un triangle extérieur de sommets $i-1,i,i+1$, avec $2 \leq i \leq n$, car si $T$ n'était pas extérieur, il empêcherait le tracé de certains triangles $i,i+1,n+2$. Par exemple, dans la figure ci-dessous, on ne peut pas construire le triangle de sommets $n,n+1,n+2$.
$$
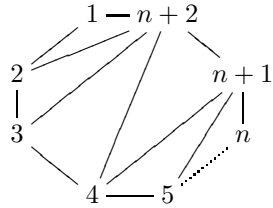
\shorthandoff{; :!?}
\xymatrix @!0 @R=0.40cm @C=0.5cm
{
&&1\ar@{-}[lldd] \ar@{-}[rr]&&n+2\ar@{-}[rrdd]\ar@{-}[lllldd]\ar@{-}[lllldddd]&
\\
&&&
\\
2\ar@{-}[dd]&&&&&& n+1 \ar@{-}[dd]\ar@{-}[lldddd]\ar@{-}[lllldddd]
\\
&&&&
\\
3&&&&&& n
\\
&&&
\\
&&4\ar@{-}[rr]\ar@{-}[lluu]\ar@{-}[rruuuuuu]  &&5 \ar@{.}[rruu]
}
$$
Ainsi, $V_{n-1,n}=n-1$.
\\
\\iii) Montrons maintenant que $V_{n-2,n}=\frac{(n+1)(n-2)}{2}$. Soit $\mathcal{P}$ un polygone à $n+2$ côtés dont les sommets sont numérotés dans le sens trigonométrique de 1 à $n+2$. Pour construire une triangulation de $\mathcal{P}$ dans laquelle $n-2$ triangles utilisent le sommet $n+2$, on a deux possibilités : soit on choisit un quadrilatère extérieur n'utilisant pas le sommet $n+2$ (que l'on coupe ensuite en deux) soit on choisit deux triangles extérieurs n'utilisant pas le sommet $n+2$. Une fois ce choix fait, les $n-2$ triangles restants sont tous de la forme $i,i+1,n+2$. Comptons maintenant le nombre de possibilités offertes par chaque cas :
\begin{itemize}
\item Pour avoir un quadrilatère extérieur n'utilisant pas le sommet $n+2$, on choisit un sommet $1 \leq i \leq n-2$ et on construit le quadrilatère dont les sommets sont $i,i+1,i+2,i+3$. Ensuite, on coupe ce quadrilatère en deux en choisissant une des deux diagonales possibles (dans la figure ci-dessous $i_{k}:=i+k$).
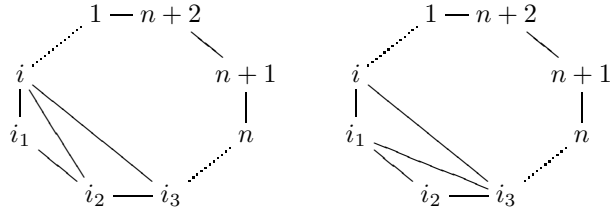
$$
\shorthandoff{; :!?}
\xymatrix @!0 @R=0.40cm @C=0.5cm
{
&&1\ar@{.}[lldd] \ar@{-}[rr]&&n+2\ar@{-}[rrdd]&
\\
&&&
\\
i\ar@{-}[dd]\ar@{-}[rrdddd]&&&&&& n+1 \ar@{-}[dd]
\\
&&&&
\\
i_{1}&&&&&& n
\\
&&&
\\
&&i_{2}\ar@{-}[rr] \ar@{-}[lluu] &&i_{3} \ar@{.}[rruu]\ar@{-}[lllluuuu]
}
\qquad
\xymatrix @!0 @R=0.40cm @C=0.5cm
{
&&1\ar@{.}[lldd] \ar@{-}[rr]&&n+2\ar@{-}[rrdd]&
\\
&&&
\\
i\ar@{-}[dd]&&&&&& n+1 \ar@{-}[dd]
\\
&&&&
\\
i_{1}\ar@{-}[rrrrdd]&&&&&& n
\\
&&&
\\
&&i_{2}\ar@{-}[rr] \ar@{-}[lluu] &&i_{3} \ar@{.}[rruu]\ar@{-}[lllluuuu]
}
$$
Cela donne au final $2(n-2)$ possibilités.
\item Pour avoir deux triangles extérieurs n'utilisant pas le sommet $n+2$, on commence par choisir un $2 \leq i \leq n-2$ et on construit le triangle de sommets $i-1,i,i+1$. Ensuite, on choisit un $i+2 \leq j \leq n$ et on construit le triangle de sommets $j-1,j,j+1$ (dans la figure ci-dessous $i_{k}:=i+k$, $j_{k}:=j+k$ et $n_{k}:=n+k$).
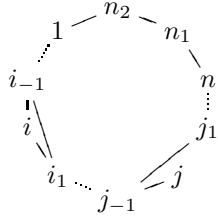
$$
\shorthandoff{; :!?}
\xymatrix @!0 @R=0.32cm @C=0.40cm
 {
&&&n_{2}\ar@{-}[lld]\ar@{-}[rrd]&
\\
&1\ar@{.}[ldd]&&&& n_{1}\ar@{-}[rdd]\\
\\
i_{-1}\ar@{-}[dd]\ar@{-}[rdddd]&&&&&&n\ar@{.}[dd]\\
\\
i\ar@{-}[rdd]&&&&&&j_{1}\ar@{-}[ldd]\ar@{-}[lllddd]\\
\\
&i_{1}\ar@{.}[rrd]&&&& j\ar@{-}[lld]\\
&&&j_{-1}&
}
$$
Cela donne au final $\sum_{i=2}^{n-2} n-i-1=(n-3)(n-1)-\frac{(n-1)(n-2)}{2}+1$ possibilités.
\\
\end{itemize}

\noindent On a donc $V_{n-2,n}=(n-3)(n-1)-\frac{(n-1)(n-2)}{2}+1+2(n-2)=
\frac{(n+1)(n-2)}{2}$.
\\
\\iv) Soit $1 \leq k \leq n$. On triangule le polygone de sommets $1,n-k+2,\ldots,n+2$ avec la méthode présentée en i). Puis, on effectue une triangulation quelconque du polygone de sommets $1,2,\ldots,n-k+1$. On obtient ainsi une triangulation où exactement $k$ triangles utilisent le sommet $n+2$. Par conséquent, $V_{k,n}>0$. 

\end{proof}

\noindent On va maintenant faire la preuve du résultat principal de cette section.

\begin{proof}[Démonstration du théorème \ref{11}]

Soient $n \geq 1$, $k \geq 1$ et $\mathcal{P}$ un polygone convexe à $n+2$ sommets, ces derniers étant numérotés de 1 à $n+2$ dans le sens trigonométrique. Notons $U_{k,n}$ le nombre de $3d$-dissections ouvertes maximales à base ouverte de $\mathcal{P}$ dont le dernier sommet utilise $k$ sous-polygones, $U_{k,0}:=0$ et $U_{k}:=\sum_{n=0}^{+\infty} U_{k,n} X^{n}$. Par le théorème \ref{23}, $V_{k,n}$ est le nombre de $3d$-dissections ouvertes maximales de $\mathcal{P}$ dont le dernier sommet est utilisé par $k$ sous-polygones.
\\
\\On commence par chercher une formule pour $U_{1}$. Pour construire une $3d$-dissection ouverte maximale à base ouverte de $\mathcal{P}$, on commence par choisir le sommet du sous-polygone de base dont le numéro $i \leq n+1$ est le plus grand, ce qui permet de construire le côté de sommets $i$ et $n+2$.

$$
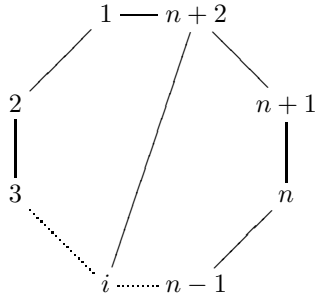
\shorthandoff{; :!?}
\xymatrix @!0 @R=0.60cm @C=0.6cm
{
&&1\ar@{-}[lldd] \ar@{-}[rr]&&n+2\ar@{-}[rrdd]\ar@{-}[lldddddd]&
\\
&&&
\\
2\ar@{-}[dd]&&&&&& n+1 \ar@{-}[dd]
\\
&&&&
\\
3&&&&&& n
\\
&&&
\\
&&i\ar@{.}[rr] \ar@{.}[lluu] &&n-1 \ar@{-}[rruu]
} 
$$

\noindent Pour avoir une $3d$-dissection ouverte maximale à base ouverte sur $\mathcal{P}$, on doit construire une $3d$-dissection ouverte maximale à base ouverte pour le polygone $i,i+1,\ldots,n+1,n+2$ et une $3d$-dissection ouverte maximale à base ouverte pour le polygone $1,2,\ldots,i,n+2$ dont le dernier sommet n'est utilisé que par un seul sous-polygone. Cela donne $U_{1,i-1}P_{n+1-i}$ constructions possibles. Puisque $i$ peut varier de 2 à $n+1$, on a $\sum_{i=2}^{n+1} U_{1,i-1}P_{n+i-1}=\sum_{j=1}^{n} U_{1,j}P_{n-j}$ possibilités. Comme $U_{1,0}=0$, on peut rajouter $U_{1,0}P_{n}$, ce qui donne pour tout $n \geq 1$:
\[P_{n}=\sum_{j=0}^{n} U_{1,j}P_{n-j}.\]

\noindent On a également $P_{0}-1=U_{1,0}P_{0}$ et on reconnaît ici la formule du produit de Cauchy de $U_{1}$ et de $P$. Plus précisément, on a $P-1=U_{1}P$. Comme $P_{0}=1$, $P$ est inversible dans $\mathbb{Z}[[X]]$. Par conséquent, $U_{1}=1-P^{-1}$.
\\
\\On s'intéresse maintenant à $U_{k}$. Pour avoir une $3d$-dissection ouverte maximale à base ouverte sur $\mathcal{P}$ dont le dernier sommet est utilisé par $k$ sous-polygones, on choisit un sommet $i$ comme dans le cas précédent, avec $i \leq n-k+2$. Ensuite, on doit construire une $3d$-dissection ouverte maximale à base ouverte pour le polygone $i,i+1,\ldots,n+1,n+2$ dont le dernier sommet est utilisé par $k-1$ sous-polygones et une $3d$-dissection ouverte maximale à base ouverte pour le polygone $1,2,\ldots,i,n+2$ dont le dernier sommet n'est utilisé que par un seul sous-polygone. On a ainsi $U_{1,i-1}U_{k-1,n+1-i}$ constructions possibles. Cela donne $\sum_{i=2}^{n-k+2} U_{1,i-1}U_{k-1,n-i+1}=\sum_{j=1}^{n-k+1} U_{1,j}U_{k-1,n-j}$ possibilités. Comme $U_{k-1,n-j}=0$ pour $j \geq n-k+2$ et $U_{1,0}=0$, on a :
\[U_{k,n}=\sum_{j=0}^{n} U_{1,j}U_{k-1n-j}.\]

\noindent Ainsi, $U_{k}=U_{1}U_{k-1}=U_{1}^{2}U_{k-2}=\ldots=U_{1}^{k}=(1-P^{-1})^{k}$.
\\
\\On va maintenant s'intéresser à $V_{1}$, en utilisant des raisonnements analogues. Pour construire une $3d$-dissection ouverte maximale de $\mathcal{P}$, on commence par choisir le sommet du sous-polygone de base dont le numéro $i \leq n+1$ est le plus grand, ce qui permet de construire le côté de sommets $i$ et $n+2$. Pour avoir une $3d$-dissection ouverte maximale sur $\mathcal{P}$, on doit construire une $3d$-dissection ouverte maximale à base ouverte pour le polygone $i,i+1,\ldots,n+1,n+2$ et une $3d$-dissection ouverte maximale pour le polygone $1,2,\ldots,i,n+2$ dont le dernier sommet n'est utilisé que par un seul sous-polygone. Cela donne $V_{1,i-1}P_{n+1-i}$ constructions possibles. Puisque $i$ peut varier de 2 à $n+1$, on a $\sum_{i=2}^{n+1} V_{1,i-1}P_{n+i-1}=\sum_{j=1}^{n} V_{1,j}P_{n-j}$ possibilités. Comme $V_{1,0}=0$, on peut rajouter $V_{1,0}P_{n}$, ce qui donne 
\[Q_{n}=\sum_{j=0}^{n} V_{1,j}P_{n-j}.\]

\noindent On a également $Q_{0}-1=V_{1,0}P_{0}$ et on reconnaît ici la formule du produit de Cauchy de $V_{1}$ et de $P$. Plus précisément, on a $Q-1=V_{1}P$, c'est-à-dire $V_{1}=(Q-1)P^{-1}$.
\\
\\En reprenant le raisonnement fait pour $V_{1}$ et celui effectué pour $U_{k}$, on trouve les formules ci-dessous :
\[V_{k}=V_{1}U_{k-1}=V_{1}U_{1}^{k-1}=[(Q-1)P^{-1}](1-P^{-1})^{k-1}.\]

\noindent Pour terminer la preuve, il ne nous reste donc plus qu'à démontrer la formule explicite de $V_{1}$ présente dans l'énoncé.
\\
\\Par le théorème \ref{25}, $Q(X)-1=\sum_{n=0}^{+\infty} X^{3n+1}P(X)^{3n+2}$. Donc, $(Q(X)-1)P^{-1}(X)=\sum_{n=0}^{+\infty} X^{3n+1}P(X)^{3n+1}$. Ainsi, pour tout $n \geq 1$, le coefficient $[X^{n}](Q(X)-1)P^{-1}(X)$ est $\sum_{j=0}^{E\left[\frac{n-1}{3}\right]} [X^{n-3j-1}]P(X)^{3j+1}$. Par la proposition \ref{26}, on a :

\[V_{1}=(Q-1)P^{-1}=\sum_{n=1}^{+\infty} \left(\sum_{j=0}^{E\left[\frac{n-1}{3}\right]} \sum_{k=0}^{E\left[\frac{n-3j-1}{3}\right]} \frac{3j+1}{n-k} {n-3j-1-2k-1 \choose k} {2n-4k-3j-2 \choose n-3k-3j-1}\right)X^{n}.\]

\end{proof}

Ainsi, en combinant les théorèmes \ref{10} et \ref{11} et en utilisant la formule du produit de Cauchy, on peut obtenir les valeurs souhaitées. De plus, la preuve que nous venons d'effectuer nous permet d'avoir immédiatement le résultat ci-dessous.

\begin{corollary}
\label{27}

Soient $n \geq 1$, $k \geq 1$ et $\mathcal{P}$ un polygone convexe à $n+2$ sommets, ces derniers étant numérotés de 1 à $n+2$ dans le sens trigonométrique.
\\
\\i) $-\tilde{P_{n}}$ est le nombre de $3d$-dissections ouvertes maximales à base ouverte de $\mathcal{P}$ dont le dernier sommet utilise un seul sous-polygone.
\\
\\ii) Le coefficient de $X^{n}$ dans la série formelle $(1-P^{-1})^{k}$ est le nombre de $3d$-dissections ouvertes maximales à base ouverte de $\mathcal{P}$ dont le dernier sommet utilise exactement $k$ sous-polygones.

\end{corollary}

\begin{remark}
{\rm Le corollaire précédent donne une interprétation combinatoire pertinente des coefficients de $1-P^{-1}$. Il serait intéressant d'avoir également une interprétation de la série formelle $1-Q^{-1}$.
}
\end{remark}

\noindent Pour terminer, cette sous-section, on va démontrer le petit résultat ci-dessous :

\begin{proposition}
\label{28}

Soient $n \geq 1$, $k \geq 1$ et $1 \leq i \leq n+2$. Il y a $V_{k,n}$ $\lambda$-quiddités de taille $n+2$ dont la $i^{\grave{e}me}$ composante est égale à $k$.

\end{proposition}

\begin{proof}

Pour tout $1 \leq j \leq n+2$, on note $\Theta_{n,j}:=\{(a_{1},\ldots,a_{n+2}) \in (\mathbb{N}^{*})^{n+2},~M_{n+2}(a_{1},\ldots,a_{n+2})=\pm Id~{\rm et}~a_{j}=k\}$. On considère les applications suivantes :
\[\begin{array}{ccccc} 
f_{i} : & \Theta_{n,n+2} & \longrightarrow & \Theta_{n,i} \\
  & (a_{1},\ldots,a_{n+2}) & \longmapsto & (a_{n+3-i},\ldots,a_{n+2},a_{1},\ldots,a_{n+2-i})  \\
\end{array},\]
\[\begin{array}{ccccc} 
g_{i} : & \Theta_{n,i} & \longrightarrow & \Theta_{n,n+2} \\
 & (a_{1},\ldots,a_{n+2}) & \longmapsto &  (a_{i+1},\ldots,a_{n+2},a_{1},\ldots,a_{i})  \\
\end{array}.\]

\noindent Les solutions de $(E_{Id})$ étant invariantes par permutations circulaires, $f_{i}$ et $g_{i}$ sont bien définies et sont des bijections réciproques. Donc, ${\rm card}(\Theta_{n,i})={\rm card}(\Theta_{n,n+2})=V_{k,n}$.

\end{proof}

\subsection{Applications numériques}
\label{num}

\hfill\break

\begin{center}
\begin{tabular}{|c|c|c|c|c|c|c|c|c|c|c|c|c|c|}
\hline
  n & 0 & 1 & 2 & 3 & 4 & 5 & 6 &7 & 8 & 9 & 10 & 11 & 12 \\
	\hline
  $\tilde{P_{n}}$ & 1 & -1 & -1 & -2 & -6 & -18 & -57 & -189 & -648 & -2278 & -8166 & -29~737 & -109~701     \rule[-7pt]{0pt}{18pt} \\
	\hline
	\end{tabular}
\end{center}

\hfill\break

\begin{center}
\begin{tabular}{|c|c|c|c|c|c|c|c|c|c|c|c|c|c|}
\hline
  n & 0 & 1 & 2 & 3 & 4 & 5 & 6 &7 & 8 & 9 & 10 & 11 & 12 \\
	\hline
  $U_{2,n}$ & 0 & 0 & 1 & 2 & 5 & 16 & 52 & 174 & 600 & 2118 & 7616 & 27~800 & 102~747    \rule[-7pt]{0pt}{18pt} \\
	\hline
	$U_{3,n}$ & 0 & 0 & 0 & 1 & 3 & 9 & 31 & 108 & 381 & 1367 & 4977 & 18~345 & 68~334    \rule[-7pt]{0pt}{18pt} \\
	\hline
	\end{tabular}
\end{center}

\hfill\break

\noindent Le tableau ci-dessous donne les premières valeurs de $V_{k,n}$.

\hfill\break

\begin{center}
\begin{tabular}{|c|c|c|c|c|c|c|c|c|c|c|c|c|c|}
\hline
  \multicolumn{1}{|c|}{\backslashbox{$k$}{\vrule width 0pt height 1.25em$n$}} & 0 & 1 & 2 & 3 & 4 & 5 & 6 &7 & 8 & 9 & 10 & 11 & 12 \\
	\hline
  1   & 0 & 1 & 1 & 2 & 6 & 19 & 62 & 209 & 726 & 2580 & 9331 & 34~229 & 127~050     \rule[-7pt]{0pt}{18pt} \\
	\hline
  2   & 0 & 0 & 1 & 2 & 5 & 16 & 53 & 180 & 627 & 2232 & 8084 & 29~690 & 110~310     \rule[-7pt]{0pt}{18pt} \\
	\hline
  3   & 0 & 0 & 0 & 1 & 3 & 9 & 31 & 109 & 388 & 1402 & 5136 & 19~035 & 71~247     \rule[-7pt]{0pt}{18pt} \\
	\hline
  4   & 0 & 0 & 0 & 0 & 1 & 4 & 14 & 52 & 194 & 724 & 2716 & 10~254 & 38~955     \rule[-7pt]{0pt}{18pt} \\
	\hline
  5   & 0 & 0 & 0 & 0 & 0 & 1 & 5 & 20 & 80 & 316 & 1235 & 4809 & 18~720     \rule[-7pt]{0pt}{18pt} \\
\hline
  6   & 0 & 0 & 0 & 0 & 0 & 0 & 1 & 6 & 27 & 116 & 484 & 1978 & 7990     \rule[-7pt]{0pt}{18pt} \\
\hline
  7   & 0 & 0 & 0 & 0 & 0 & 0 & 0 & 1 & 7 & 35 & 161 & 708 & 3021     \rule[-7pt]{0pt}{18pt} \\
\hline
  8   & 0 & 0 & 0 & 0 & 0 & 0 & 0 & 0 & 1 & 8 & 44 & 216 & 999     \rule[-7pt]{0pt}{18pt} \\
\hline
  9   & 0 & 0 & 0 & 0 & 0 & 0 & 0 & 0 & 0 & 1 & 9 & 54 & 282     \rule[-7pt]{0pt}{18pt} \\
\hline
  10   & 0 & 0 & 0 & 0 & 0 & 0 & 0 & 0 & 0 & 0 & 1 & 10 & 65     \rule[-7pt]{0pt}{18pt} \\
\hline
  11   & 0 & 0 & 0 & 0 & 0 & 0 & 0 & 0 & 0 & 0 & 0 & 1 & 11     \rule[-7pt]{0pt}{18pt} \\
\hline
  12   & 0 & 0 & 0 & 0 & 0 & 0 & 0 & 0 & 0 & 0 & 0 & 0 & 1     \rule[-7pt]{0pt}{18pt} \\
\hline
	
\end{tabular}
\end{center}

\hfill\break

Notons que, pour tout $n \geq 1$, on retrouve, en sommant toutes les valeurs situées sur la colonne indexée par $n$, la valeur de $Q_{n}$.

\section{Nombre de solutions de $(E_{S})$ et de $(E_{T})$}
\label{SetT} 

\noindent Nous allons maintenant nous intéresser aux cas des générateurs du groupe modulaire.

\subsection{Description combinatoire des solutions}
\label{des}

L'objectif de cette sous-partie est de rappeler les modèles combinatoires dont on dispose pour les solutions de $(E_{S})$ et de $(E_{T})$ et de pointer l'inadéquation de ces derniers pour la résolution du problème de comptage qui nous occupe.

\begin{definition}[\cite{M1}, définition 2.1]
\label{31}

Soient $n \in \mathbb{N}^{*}$, $n \geq 5$ et $\mathcal{P}$ un polygone convexe à $n+1$ sommets. Une 3$d$-dissection échancrée de $\mathcal{P}$ est une décomposition de $\mathcal{P}$ en sous-polygones par des diagonales ne se croisant qu'aux sommets de $\mathcal{P}$ et vérifiant les conditions suivantes : 
\begin{itemize}
\item un seul des sous-polygones est un quadrilatère;
\item tous les autres sous-polygones intervenant dans la décomposition possèdent un nombre de sommets égal à un multiple de $3$;
\item le quadrilatère a exactement deux de ses côtés qui sont des côtés de $\mathcal{P}$ et ces deux côtés possèdent un sommet en commun. 
\end{itemize}
\noindent On numérote ce sommet $0$ puis on numérote les autres sommets dans le sens trigonométrique. Pour chaque $i \in \llbracket 1~;~ n \rrbracket$, on note $a_{i}$ le nombre de sous-polygones possédant un nombre de sommets égal à un multiple de $3$ utilisant le sommet $i$. $(a_{1},\ldots,a_{n})$ est la quiddité de la 3$d$-dissection échancrée de $\mathcal{P}$.

\end{definition}

\begin{definition}[\cite{M1}, définition 2.3]
\label{32}

Soient $n \in \mathbb{N}^{*}$, $n \geq 3$ et $\mathcal{P}$ un polygone convexe à $n+2$ sommets. Une 3$d$-dissection coiffée de $\mathcal{P}$ est une décomposition de $\mathcal{P}$ en sous-polygones par des diagonales ne se croisant qu'aux sommets de $\mathcal{P}$ et vérifiant les conditions suivantes : 
\begin{itemize}
\item tous les sous-polygones intervenant dans la décomposition possèdent un nombre de sommets égal à un multiple de $3$;
\item $\mathcal{P}$ possède un triangle extérieur (c'est-à-dire que deux de ses côtés sont des côtés de $\mathcal{P}$) auquel on affecte le poids -1;
\item tous les autres sous-polygones reçoivent le poids 1;
\item le triangle de poids -1, noté $T$, a un côté en commun avec un triangle de poids 1, noté $T'$, et, un des côtés de $T'$ est un côté de $\mathcal{P}$.
\end{itemize}
\noindent Le sommet appartenant uniquement à $T$ est numéroté $0$ et celui appartenant uniquement à $T$ et à $T'$ est numéroté $-1$. On numérote les autres sommets de telle façon que le sommet numéroté $n$ soit adjacent au sommet numéroté $-1$. Pour chaque $i \in \llbracket 1~;~ n \rrbracket$, on note $a_{i}$ la somme des poids des sous-polygones utilisant le sommet $i$. $(a_{1},\ldots,a_{n})$ est la quiddité de la 3$d$-dissection coiffée de $P$.

\end{definition}

\begin{examples}
{\rm On donne ci-dessous une $3d$-dissection échancrée avec sa quiddité et une $3d$-dissection coiffée avec sa quiddité.}

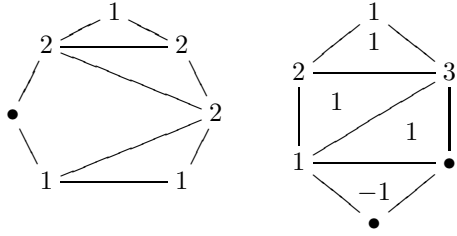
$$
\shorthandoff{; :!?}
\xymatrix @!0 @R=0.45cm @C=0.45cm
{
&&&1\ar@{-}[rrd]\ar@{-}[lld]
\\
&2\ar@{-}[ldd]\ar@{-}[rrrrrdd]\ar@{-}[rrrr]&&&& 2\ar@{-}[rdd]&\\
\\
\bullet\ar@{-}[rdd]&&&&&& 2\ar@{-}[ldd]\\
\\
&1\ar@{-}[rrrr]\ar@{-}[rrrrruu]&&&&1
}
\qquad
\xymatrix @!0 @R=0.40cm @C=0.5cm
{
&&1\ar@{-}[rrdd]\ar@{-}[lldd]&
\\
&& 1
\\
2\ar@{-}[ddd]\ar@{-}[rrrr]&& && 3\ar@{-}[ddd]
\\
&1
\\
&&& 1
\\
1 \ar@{-}[rrrruuu] \ar@{-}[rrrr] &&&& \bullet
\\
&& -1
\\
&&\bullet\ar@{-}[lluu]\ar@{-}[rruu]
}
$$
\end{examples}

\noindent Ces objets combinatoires sont reliés aux équations $(E_{S})$ et $(E_{T})$ par le résultat suivant.

\begin{theorem}[\cite{M1} Théorèmes 2.2 et 2.4]
\label{33}

i) Soit $n \geq 5$. Tout $n$-uplet d'entiers strictement positifs solution de $(E_{S})$ est la quiddité d'une 3$d$-dissection échancrée d'un polygone convexe à $n+1$ sommets et réciproquement.
\\
\\ii) Soit $n \geq 3$.
\begin{itemize}
\item $(a_{1},\ldots,a_{n-1},1)$ est solution de $(E_{T})$ si et seulement si $(a_{1},\ldots,a_{n-1})$ est la quiddité d'une 3$d$-dissection échancrée d'un polygone convexe à $n$ sommets.
\item $(a_{1},\ldots,a_{n})$ avec $a_{n} \geq 2$ est solution de $(E_{T})$ si et seulement si $(a_{1},\ldots,a_{n})$ est la quiddité d'une 3$d$-dissection coiffée d'un polygone convexe à $n+2$ sommets.
\\
\end{itemize}

\end{theorem}

Notons que contrairement aux solutions de $(E_{Id})$, les solutions de $(E_{S})$ et de $(E_{T})$ ne sont pas invariantes par permutations circulaires. Par ailleurs, les solutions de $(E_{S})$ sont invariantes par retournements mais celles de $(E_{T})$ ne le sont pas. On notera également que $(E_{S})$ n'a pas de solution de taille inférieure à 4 et que $(E_{T})$ n'a pas de solution de taille inférieure à 2. Ainsi, $\mathcal{S}_{0}=\mathcal{S}_{1}=\mathcal{S}_{2}=0$ et $\mathcal{T}_{0}=0$.
\\
\\ \indent Soient $n \geq 3$, $E_{n}$ le nombre de $3d$-dissections échancrées d'un polygone convexe à $n+2$ sommets (avec $E_{3}=0$) et $F_{n}$ le nombre de $3d$-dissections coiffées d'un polygone convexe à $n+2$ sommets. À partir de la formule de $D_{n}$ (voir section \ref{dissec}), qui donne le nombre de $3d$-dissections d'un polygone convexe à $n+2$ sommets, on peut assez facilement obtenir une expression de $E_{n}$ et de $F_{n}$. Soit $\mathcal{P}$ un polygone convexe à $n+2$ sommets. Pour construire une $3d$-dissection échancrée de $\mathcal{P}$, on commence par choisir le sommet 0, ce qui nous fait $n+2$ choix. Ensuite, on numérote les autres sommets et on choisit le sommet $3 \leq k \leq n-1$ qui permet de tracer le quadrilatère de sommets $0,1,k,n+1$. Pour chaque $k$, on effectue une $3d$-dissection des polygones de sommets $1,2,\ldots,k$ et $k,k+1,\ldots,n+1$, ce qui donne $D_{k-2}D_{n-k}$ possibilités. Ainsi, $E_{n}=(n+2)\sum_{k=3}^{n-1} D_{k-2}D_{n-k}$. Pour construire une $3d$-dissection coiffée de $\mathcal{P}$, on choisit un quadrilatère extérieure de $\mathcal{P}$, ce qui fait $n+2$ choix, et on coupe celui-ci en deux, comme dans la figure ci-dessous : 

$$
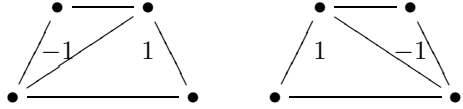
\shorthandoff{; :!?}
\xymatrix @!0 @R=0.6cm @C=0.6cm
{
&\bullet\ar@{-}[rr]\ar@{-}[ldd]&&\bullet\ar@{-}[rdd]\ar@{-}[llldd]
\\&-1&&1
\\
\bullet\ar@{-}[rrrr]&&&&\bullet
}
\qquad
\xymatrix @!0 @R=0.6cm @C=0.6cm
{
&\bullet\ar@{-}[rr]\ar@{-}[ldd]\ar@{-}[rrrdd]&&\bullet\ar@{-}[rdd]
\\&1&&-1
\\
\bullet\ar@{-}[rrrr]&&&&\bullet
}
$$

\noindent Ensuite, on effectue une $3d$-dissection du polygone obtenu en supprimant le quadrilatère de $\mathcal{P}$. Cela donne $F_{n}=2(n+2)D_{n-2}$.

\hfill\break

\begin{center}
\begin{tabular}{|c|c|c|c|c|c|c|c|c|c|c|}
\hline
  n & 3 & 4 & 5 & 6 & 7 & 8 & 9 & 10 & 11 & 12 \\
	\hline
  $D_{n}$ & 5 & 15 & 49 & 168 & 595 & 2160 & 7997 & 30~083 & 114~660 & 441~840       \rule[-7pt]{0pt}{18pt} \\
	\hline
	$E_{n}$ & 0 & 6 & 28 & 112 & 450 & 1830 & 7502 & 30~924 & 128~050 & 532~350       \rule[-7pt]{0pt}{18pt} \\
	\hline
	$F_{n}$ & 10 & 24 & 70 & 240 & 882 & 3360 & 13~090 & 51~840 & 207~922 & 842~324      \rule[-7pt]{0pt}{18pt} \\
	\hline
	\end{tabular}
\end{center}

\hfill\break

Malheureusement, on ne peut pas utiliser ces descriptions combinatoires pour compter le nombre de solutions de $(E_{S})$ et de $(E_{T})$. En effet, puisque deux $3d$-dissections différentes peuvent avoir la même quiddité, il en va de même pour les $3d$-dissections échancrées ou coiffées. Par ailleurs, la situation dans le cas échancré est encore plus complexe que celle rencontrée dans la section précédente car il est possible d'avoir deux $3d$-dissections échancrées d'un même polygone, avec le sommet 0 placé à la même position, donnant la même quiddité alors que leur quadrilatère sont positionnés de façon différente, comme l'illustre les deux décompositions de l'hendécagone ci-dessous :

$$
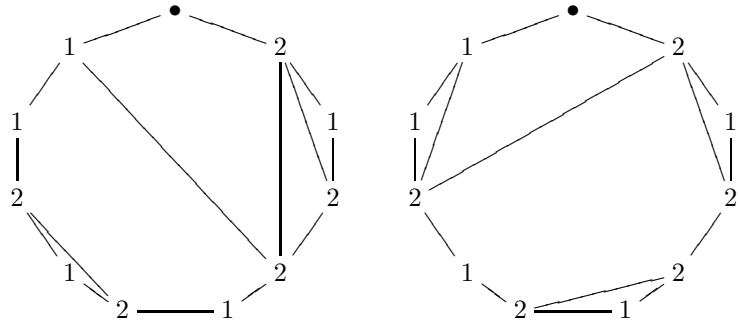
\shorthandoff{; :!?}
\xymatrix @!0 @R=0.5cm @C=0.7cm
{
&&&\bullet\ar@{-}[lld]\ar@{-}[rrd]
\\
&1\ar@{-}[ldd]\ar@{-}[rrrrdddddd]&&&& 2\ar@{-}[rdd]\ar@{-}[rdddd]\ar@{-}[dddddd]\\
\\
1\ar@{-}[dd]&&&&&&1\ar@{-}[dd]\\
\\
2\ar@{-}[rdd]\ar@{-}[rrddd]&&&&&&2\ar@{-}[ldd]\\
\\
&1\ar@{-}[rd]\ar@{-}&&&&2 \ar@{-}[ld]\\
&&2\ar@{-}[rr]&&1&
}
\qquad
\xymatrix @!0 @R=0.5cm @C=0.7cm
{
&&&\bullet\ar@{-}[lld]\ar@{-}[rrd]
\\
&1\ar@{-}[ldd]\ar@{-}[ldddd]&&&& 2\ar@{-}[rdd]\ar@{-}[rdddd]\ar@{-}[llllldddd]\\
\\
1\ar@{-}[dd]&&&&&&1\ar@{-}[dd]\\
\\
2\ar@{-}[rdd]&&&&&&2\ar@{-}[ldd]\\
\\
&1\ar@{-}[rd]\ar@{-}&&&&2 \ar@{-}[ld]\ar@{-}[llld]\\
&&2\ar@{-}[rr]&&1&
}
$$

Par conséquent, pour effectuer le comptage souhaité, on n'utilisera pas les descriptions combinatoires présentées ici.

\subsection{Preuve des formules de comptage} L'objectif de cette sous-partie est d'utiliser les résultats prouvés dans la section \ref{DCQ} pour pouvoir dénombrer les solutions de $(E_{S})$ et de $(E_{T})$.

\begin{proof}[Démonstration du théorème \ref{12}]

i) Soient $n \geq 3$, $(a_{1},\ldots,a_{n+2}) \in (\mathbb{N}^{*})^{n+2}$ et $\epsilon \in \{-1,1\}$. On a 
\begin{eqnarray*}
M_{n+2}(a_{1},\ldots,a_{n+2})=\epsilon S & \Longleftrightarrow & SM_{n+2}(a_{1},\ldots,a_{n+2})=-\epsilon Id \\
                                         & \Longleftrightarrow & T^{0}SM_{n+2}(a_{1},\ldots,a_{n+2})=-\epsilon Id \\
                                         & \Longleftrightarrow & M_{n+3}(a_{1},\ldots,a_{n+2},0)=-\epsilon Id \\
												                 & \Longleftrightarrow & M_{n+3}(a_{2},\ldots,a_{n+2},0,a_{1})=-\epsilon Id \\
																			   & \Longleftrightarrow & M_{n+2}(a_{2},\ldots,a_{1}+a_{n+2})=\epsilon Id. \\
\end{eqnarray*}

\noindent La dernière équivalence provient de l'égalité $M_{3}(a,0,b)=-M_{1}(a+b)$. Par conséquent, $(a_{1},\ldots,a_{n+2})$ est une solution de $(E_{S})$ si et seulement si $(a_{2},\ldots,a_{n+1},a_{1}+a_{n+2})$ est une $\lambda$-quiddité.
\\
\\De plus, si $(a_{1},\ldots,a_{n+2}) \in (\mathbb{N}^{*})^{n+2}$ est une solution de $(E_{S})$, on a $2 \leq a_{1}+a_{n+2} \leq n-1$. En effet, par ce qui précède, $(a_{2},\ldots,a_{n+1}, a_{1}+a_{n+2})$ est une $\lambda$-quiddité de taille $n+1$. Or, par le théorème \ref{21bis}, cette $\lambda$-quiddité est la quiddité d'une $3d$-dissection d'un polygone convexe $\mathcal{P}$ à $n+1$ sommets. Comme chaque sommet de $\mathcal{P}$ ne peut être utilisé que par au plus $n-1$ sous-polygones, on a nécessairement $a_{1}+a_{n+2} \leq n-1$. Par ailleurs, $a_{1},a_{n+2} \geq 1$. Par conséquent, on a bien l'encadrement souhaité.
\\
\\Pour compter le nombre de solutions de $(E_{S})$, on va donc s'intéresser aux solutions dont la somme du premier et du dernier terme est fixée. Plus précisément, on note $X_{n}$ l'ensemble de solutions de $(E_{S})$ de taille $n+2$ et on pose pour tout $2 \leq d \leq n-1$ et pour tout $1 \leq k \leq d-1$, $Y_{n,d}:=\{(a_{1},\ldots,a_{n+2}) \in X_{n},~a_{1}+a_{n+2}=d\}$ et $Z_{n,d,k}:=\{(a_{1},\ldots,a_{n+2}) \in X_{n},~a_{1}=k~{\rm et}~a_{n+2}=k-d\}$. Par ce qui précède, on a :
\[X_{n}=\bigsqcup_{d=2}^{n-1} Y_{n,d}=\bigsqcup_{d=2}^{n-1} \bigsqcup_{k=1}^{d-1} Z_{n,d,k}.\]
\noindent Par conséquent, \[\mathcal{S}_{n}=\sum_{d=2}^{n-1} \sum_{k=1}^{d-1} {\rm card}(Z_{n,d,k}).\]

\noindent On va donc chercher à calculer ${\rm card}(Z_{n,d,k})$. Commençons par remarquer que cet ensemble est non vide pour tout $2 \leq d \leq n-1$ et pour tout $1 \leq k \leq d-1$. En effet, notons $O_{n,d}$ l'ensemble des $\lambda$-quiddités de taille $n+2$ dont la dernière composante est égale à $d$. On sait que $O_{n-1,d} \neq \emptyset$ (voir proposition \ref{26bis}). Soit $(a_{1},\ldots,a_{n},k+(d-k)) \in O_{n-1,d}$. Par ce qui précède, on a $(k,a_{1},\ldots,a_{n},d-k) \in Z_{n,d,k}$. Considérons maintenant les deux applications ci-dessous :
\[\begin{array}{ccccc} 
\alpha_{n,d,k} : & Z_{n,d,k} & \longrightarrow & O_{n-1,d} \\
  & (k,a_{2},\ldots,a_{n+1},d-k) & \longmapsto & (a_{2},\ldots,a_{n+1},d)  \\
\end{array},\] \[\begin{array}{ccccc} 
\beta_{n,d,k} : & O_{n-1,d} & \longrightarrow & Z_{n,d,k} \\
 & (a_{1},\ldots,a_{n},d) & \longmapsto &  (k,a_{1},\ldots,a_{n},d-k)  \\
\end{array}.\]

\noindent $\alpha_{n,d,k}$ et $\beta_{n,d,k}$ sont bien définies et sont des bijections réciproques. Par conséquent, 
\[{\rm card}(Z_{n,d,k})={\rm card}(O_{n-1,d})=V_{d,n-1}.\] 
\noindent Donc, pour tout $n \geq 3$, on a :
\[\mathcal{S}_{n}=\sum_{d=2}^{n-1} \sum_{k=1}^{d-1} V_{d,n-1}=\sum_{d=2}^{n-1} (d-1) V_{d,n-1}.\]

\noindent ii) Soient $n \geq 1$, $(a_{1},\ldots,a_{n+2}) \in (\mathbb{N}^{*})^{n+2}$ et $\epsilon \in \{-1,1\}$. On suppose d'abord $a_{n+2}=1$. On a :
\begin{eqnarray*}
M_{n+2}(a_{1},\ldots,a_{n+1},1)=\epsilon T & \Longleftrightarrow & T^{-1}M_{n+2}(a_{1},\ldots,a_{n+1},1)=\epsilon Id \\
                                           & \Longleftrightarrow & M_{n+2}(a_{1},\ldots,a_{n+1},0)=\epsilon Id \\
																			     & \Longleftrightarrow & M_{n+1}(a_{1},\ldots,a_{n+1})=-\epsilon S. \\
\end{eqnarray*}

\noindent Ainsi, il y a $\mathcal{S}_{n-1}$ solutions de $(E_{T})$ de taille $n+2$ de la forme $(a_{1},\ldots,a_{n+1},1)$.
\\
\\On suppose maintenant $a_{n+2} \geq 2$. On a :

\begin{eqnarray*}
M_{n+2}(a_{1},\ldots,a_{n+2})=\epsilon T & \Longleftrightarrow & T^{-1}M_{n+2}(a_{1},\ldots,a_{n+2})=\epsilon Id. \\
                                         & \Longleftrightarrow & M_{n+2}(a_{1},\ldots,a_{n+2}-1)=\epsilon Id. \\
\end{eqnarray*}

\noindent Par conséquent, il y a $Q_{n}$ solutions de $(E_{T})$ de taille $n+2$ de la forme $(a_{1},\ldots,a_{n+2})$ avec $a_{n+2} \geq 2$.
\\
\\Ainsi, pour tout $n \geq 1$, $\mathcal{T}_{n}=\mathcal{S}_{n-1}+Q_{n}$.

\end{proof}

\begin{remark}
{\rm À l'instar de ce qui a été fait pour les $\lambda$-quiddités, on peut s'intéresser aux solutions de $(E_{S})$ et de $(E_{T})$ dont la dernière composante est fixée. La preuve développée ci-dessus et les résultats de la section précédente permettent de résoudre rapidement cette question. 
\begin{itemize}
\item La dernière composante d'une solution de $(E_{S})$ de taille $n+2$ est comprise entre $1$ et $n-2$. Pour tout $1 \leq d \leq n-2$, l'ensemble des solutions de $(E_{S})$ de taille $n+2$ dont la dernière composante est égale à $d$ est en bijection avec l'ensemble des solutions de $(E_{Id})$ de taille $n+1$ dont la dernière composante est strictement supérieure à $d$. Aussi, il y a $\sum_{k=d+1}^{n-1} V_{k,n-1}$ solutions de $(E_{S})$ de taille $n+2$ dont la dernière composante est égale à $d$. 
\item La dernière composante d'une solution de $(E_{T})$ de taille $n+2$ est comprise entre $1$ et $n+1$. Il y a $\mathcal{S}_{n-1}$ solutions de $(E_{T})$ de taille $n+2$ dont la dernière composante est égale à 1. Pour tout $2 \leq d \leq n+1$, il y a $V_{d-1,n}$ solutions de $(E_{T})$ de taille $n+2$ dont la dernière composante est égale à d. 
\end{itemize}
}
\end{remark}

\subsection{Extension à d'autres matrices}
\label{autre}

Les techniques mises en œuvre pour démontrer les formules de comptage pour $S$ et $T$ peuvent être adaptées pour d'autres éléments du groupe modulaire parfois utilisés comme générateur. Grâce à celles-ci, on va démontrer le résultat ci-dessous :

\begin{theorem}
\label{34}

Soit $n \geq 1$. On note $u_{n}$ le nombre de solutions de taille $n+2$ de $(E_{T^{-1}})$, $v_{n}$ le nombre de solutions de taille $n+2$ de $(E_{TS})$, $w_{n}$ le nombre de solutions de taille $n+2$ de $(E_{ST})$, $x_{n}$ le le nombre de solutions de taille $n+2$ de $(E_{TSTS})$ et $y_{n}$ le nombre de solutions de taille $n+2$ de $(E_{STST})$. On a :
\\
\\i) $u_{n}=Q_{n}-V_{1,n}$.
\\
\\ii) $v_{1}=0$ et, pour $n \geq 2$, $v_{n}=Q_{n-1}+\mathcal{S}_{n}$.
\\
\\iii) $w_{n}=Q_{n}-2 \left(\sum_{k=1}^{n} W_{1,k,n}\right)+W_{1,1,n}$, avec $W_{1,1,0}=0$, $W_{1,1,1}=1$ et $W_{1,1}(X)=\sum_{n=0}^{+\infty} W_{1,1,n} X^{n}=X+\sum_{n=2}^{+\infty}\left(\sum_{j=1}^{E\left[\frac{n-1}{3}\right]} \sum_{k=0}^{E\left[\frac{n-3j-1}{3}\right]} \frac{3j}{n-1-k} {n-3j-2k-2 \choose k}{2n-3j-4k-3 \choose n-3j-3k-1}\right) X^{n}$ et, pour $k \geq 1$, $\sum_{n=0}^{+\infty} W_{1,k,n} X^{n}=W_{1,1}(X)(1-P^{-1}(X))^{k-1}$.
\\
\\iv) $x_{n}=V_{1,n+1}$.
\\
\\v) $y_{1}=y_{2}=y_{3}=0$ et, pour $n \geq 4$, $y_{n}=\sum_{d=3}^{n-1} (d-2) V_{d,n-1}$.

\end{theorem}

\begin{remarks}
{\rm i) On a $(TS)^{3}=(ST)^{3}=-Id$.
\\
\\ii) $(E_{T^{-1}})$ n'a pas de solution de taille 1, 2 ou 3. $(1)$ est l'unique solution de $(E_{TS})$ de taille 1 et n'a pas de solution de taille 2 ou 3. $(E_{ST})$ et $(E_{STST})$ n'ont pas de solution de taille 1, 2 ou 3. $(E_{TSTS})$ n'a pas de solution de taille 1 mais possède une solution de taille 2, $(1,1)$, et une solution de taille 3, $(2,1,2)$.
}
\end{remarks}

\begin{proof}

i) Soient $n \geq 1$, $(a_{1},\ldots,a_{n+2}) \in (\mathbb{N}^{*})^{n+2}$ et $\epsilon \in \{-1,1\}$. On a 
\begin{eqnarray*}
M_{n+2}(a_{1},\ldots,a_{n+2})=\epsilon T^{-1} & \Longleftrightarrow & TM_{n+2}(a_{1},\ldots,a_{n+2})=\epsilon Id \\
                                              & \Longleftrightarrow & M_{n+2}(a_{1},\ldots,a_{n+2}+1)=\epsilon Id. \\
\end{eqnarray*}

\noindent Par conséquent, l'ensemble des solutions de taille $n+2$ de $(E_{T^{-1}})$ est en bijection avec l'ensemble des solutions de taille $n+2$ de $(E_{Id})$ dont la dernière composante est supérieure ou égale à 2. Ainsi, $u_{n}=Q_{n}-V_{1,n}$.
\\
\\ii) Soient $n \geq 2$, $(a_{1},\ldots,a_{n+2}) \in (\mathbb{N}^{*})^{n+2}$ et $\epsilon \in \{-1,1\}$. On suppose pour commencer que $a_{n+2}=1$. 
\begin{eqnarray*}
M_{n+2}(a_{1},\ldots,a_{n+1},1)=\epsilon TS & \Longleftrightarrow & TS M_{n+1}(a_{1},\ldots,a_{n+1})=\epsilon TS \\
                                            & \Longleftrightarrow & M_{n+1}(a_{1},\ldots,a_{n+1})=\epsilon Id. \\
\end{eqnarray*}

\noindent Par conséquent, il y a $Q_{n-1}$ solutions de $(E_{TS})$ de taille $n+2$ de la forme $(a_{1},\ldots,a_{n+1},1)$. On suppose maintenant $a_{n+2}>1$.
\begin{eqnarray*}
M_{n+2}(a_{1},\ldots,a_{n+2})=\epsilon TS & \Longleftrightarrow & T^{-1}M_{n+2}(a_{1},\ldots,a_{n+2})=\epsilon S \\
                                          & \Longleftrightarrow & M_{n+2}(a_{1},\ldots,a_{n+2}-1)=\epsilon S. \\
\end{eqnarray*}

\noindent Par conséquent, il y a une bijection entre l'ensemble des solutions de $(E_{TS})$ de taille $n+2$ de la forme $(a_{1},\ldots,a_{n+1},a_{n+2})$ avec $a_{n+2}>1$ et l'ensemble des solutions de $(E_{S})$ de taille $n+2$. Donc, il y a $\mathcal{S}_{n}$ solutions de $(E_{TS})$ de cette forme.
\\
\\Ainsi, $w_{n}=Q_{n-1}+\mathcal{S}_{n}$.
\\
\\iii) Soient $n \geq 1$, $(a_{1},\ldots,a_{n+2}) \in (\mathbb{N}^{*})^{n+2}$ et $\epsilon \in \{-1,1\}$. On a :
\begin{eqnarray*}
M_{n+2}(a_{1},\ldots,a_{n+2})=\epsilon ST & \Longleftrightarrow & M_{n+1}(a_{2},\ldots,a_{n+2})T^{a_{1}}S(STST)=\epsilon ST(STST) \\
                                          & \Longleftrightarrow & -M_{n+1}(a_{2},\ldots,a_{n+2})T^{a_{1}}TST=-\epsilon Id \\
																					& \Longleftrightarrow & M_{n+1}(a_{2},\ldots,a_{n+2})T^{a_{1}+1}ST=\epsilon Id \\
																					& \Longleftrightarrow & SM_{n+1}(a_{2},\ldots,a_{n+2})T^{a_{1}+1}STS=\epsilon S \times S \\
																					& \Longleftrightarrow & T^{0}SM_{n+1}(a_{2},\ldots,a_{n+2})T^{a_{1}+1}STS=-\epsilon Id \\
																					& \Longleftrightarrow & M_{n+4}(1,a_{1}+1,a_{2},\ldots,a_{n+2},0)=-\epsilon Id \\
																					& \Longleftrightarrow & M_{n+4}(a_{1}+1,a_{2},\ldots,a_{n+2},0,1)=-\epsilon Id \\
																					& \Longleftrightarrow & -M_{n+2}(a_{1}+1,a_{2},\ldots,a_{n+2}+1)=-\epsilon Id \\
																					& \Longleftrightarrow & M_{n+2}(a_{1}+1,a_{2},\ldots,a_{n+2}+1)=\epsilon Id. \\
\end{eqnarray*}

\noindent On en déduit que $w_{n}$ est égal au nombre de $\lambda$-quiddités de taille $n+2$ dont la première et la dernière composante sont différentes de 1. Pour pouvoir compter concrètement les solutions de $(E_{ST})$, on doit donc dénombrer les solutions de $(E_{Id})$ vérifiant la condition souhaitée. Pour cela, on note pour, $n \geq 1$ et $k,l \geq 1$, $W_{k,l,n}$ le nombre de $\lambda$-quiddités de taille $n+2$ dont la première composante vaut $k$ et la dernière composante est égale à $l$. Comme les solutions de $(E_{Id})$ sont invariantes par retournements, on a $W_{l,k,n}=W_{k,l,n}$. De plus, les composantes d'une solution de $(E_{Id})$ de taille $n+2$ sont comprises entre $1$ et $n$. Il y a donc 
\[2 \left(\sum_{k=1}^{n} W_{1,k,n}\right)-W_{1,1,n}\]
 $\lambda$-quiddités de taille $n+2$ dont la première ou la dernière composante sont égales 1. Ainsi,
\[w_{n}=Q_{n}-2 \left(\sum_{k=1}^{n} W_{1,k,n}\right)+W_{1,1,n}.\]

\noindent On pose maintenant $W_{k,l,0}:=0$ et $W_{k,l}(X):=\sum_{n=0}^{+\infty} W_{k,l,n} X^{n}$. Soient $n \geq 1$, $k,l \geq 1$ et $\mathcal{P}$ un polygone convexe à $n+2$ sommets, ces derniers étant numérotés de 1 à $n+2$ dans le sens trigonométrique. Par le théorème \ref{23}, $W_{k,l,n}$ est le nombre de $3d$-dissections ouvertes maximales de $\mathcal{P}$ pour lesquels le sommet 1 est utilisé par $k$ sous-polygones et le sommet $n+2$ est utilisé par $l$ sous-polygones.
\\
\\On commence par chercher une formule pour $W_{1,1}$. Pour construire une $3d$-dissection ouverte maximale de $\mathcal{P}$, on commence par choisir le sommet du sous-polygone de base dont le numéro $2 \leq i \leq n+1$ est le plus petit, ce qui permet de construire le côté de sommets $1$ et $i$. Puis, on choisit le sommet du sous-polygone de base dont le numéro $i \leq j \leq n+1$ est le plus grand, ce qui permet de construire le côté de sommets $j$ et $n+2$ (notons que l'on peut avoir $i=j$ si le sous-polygone de base est un triangle).

$$
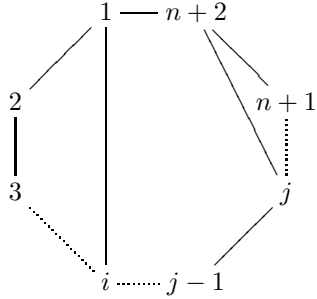
\shorthandoff{; :!?}
\xymatrix @!0 @R=0.60cm @C=0.6cm
{
&&1\ar@{-}[lldd]\ar@{-}[dddddd] \ar@{-}[rr]&&n+2\ar@{-}[rrdd]\ar@{-}[rrdddd]&
\\
&&&
\\
2\ar@{-}[dd]&&&&&& n+1 \ar@{.}[dd]
\\
&&&&
\\
3&&&&&& j
\\
&&&
\\
&&i\ar@{.}[rr] \ar@{.}[lluu] &&j-1 \ar@{-}[rruu]
} 
$$

\noindent Pour avoir une $3d$-dissection ouverte maximale sur $\mathcal{P}$, on doit construire une $3d$-dissection ouverte maximale à base ouverte pour le polygone $1,\ldots,i$, une $3d$-dissection ouverte maximale à base ouverte pour le polygone $j,\ldots,n+2$ et une $3d$-dissection ouverte maximale dont le premier et le dernier sommet ne sont utilisés que par un seul sous-polygone pour le polygone $1,i,\ldots,j,n+2$. Cela donne $P_{i-2}W_{1,1,j-i+1}P_{n+1-j}$ constructions possibles. Puisque $i$ peut varier de 2 à $n+1$ et $j$ de $i$ à $n+1$, on a 
\[\sum_{i=2}^{n+1} \sum_{j=i}^{n+1} P_{i-2}W_{1,1,j-i+1}P_{n+1-j}=\sum_{i=0}^{n-1} \sum_{j=i+2}^{n+1} P_{i}W_{1,1,j-i-1}P_{n+1-j}\] \noindent possibilités. En faisant le changement de variable $k:=j-i-1$ dans la deuxième somme, on obtient $\sum_{i=0}^{n-1} P_{i}\left(\sum_{k=1}^{n-i} W_{1,1,k}P_{(n-i)-k}\right)$ possibilités. Comme $W_{1,1,0}=0$, on peut rajouter pour tout $0 \leq i \leq n-1$, $P_{i} W_{1,1,0}P_{n-i}$ et $P_{n} W_{1,1,0}P_{0}$. Par conséquent, il y a 
\[\sum_{i=0}^{n} P_{i}\left(\sum_{k=0}^{n-i} W_{1,1,k}P_{(n-i)-k}\right)\]
\noindent façons de construire une $3d$-dissection ouverte maximale de $P$.
\\
\\On a également $Q_{0}-1=P_{0}W_{1,1,0}P_{0}$ et on reconnaît ici la formule du produit de Cauchy des trois séries $P$, $W_{1,1}$ et $P$. Plus précisément, on a $Q-1=P~W_{1,1}~P$. Comme $P_{0}=1$, $P$ est inversible dans $\mathbb{Z}[[X]]$. Par conséquent, $W_{1,1}=P^{-1}(Q-1)P^{-1}$.
\\
\\Par le théorème \ref{25}, $Q(X)=1+X~P(X)^{2}+X^{4}P(X)^{5}+X^{7}P(X)^{8}+\ldots$. Donc, 
\[W_{1,1}(X)=X+X^{4}P(X)^{3}+X^{7}P(X)^{6}+\ldots=\sum_{j=0}^{+\infty} X^{3j+1}P(X)^{3j}.\]
\noindent Donc, on a $W_{1,1,1}=1$ et, par la proposition \ref{26}, on a, pour $n \geq 2$ :
\begin{eqnarray*}
W_{1,1,n} &=& \sum_{j=0}^{E\left[\frac{n-1}{3}\right]} [X^{n-3j-1}]P(X)^{3j} \\
          &=& \sum_{j=1}^{E\left[\frac{n-1}{3}\right]} \sum_{k=0}^{E\left[\frac{n-3j-1}{3}\right]} \frac{3j}{n-1-k} {n-3j-2k-2 \choose k}{2n-3j-4k-3 \choose n-3j-3k-1}.
\end{eqnarray*}

\noindent On s'intéresse maintenant à $W_{1,k}$ pour $k \geq 2$. Pour avoir une $3d$-dissection ouverte maximale de $\mathcal{P}$ dont le premier sommet est utilisé par un seul sous-polygone et dont le dernier sommet est utilisé par $k$ sous-polygones, on commence par choisir le sommet du sous-polygone de base dont le numéro $2 \leq i \leq n-k+2$ est le plus grand, ce qui permet de construire le côté $i,n+2$. Ensuite, on doit construire une $3d$-dissection ouverte maximale à base ouverte pour le polygone $i,i+1,\ldots,n+1,n+2$ dont le dernier sommet est utilisé par $k-1$ sous-polygones et une $3d$-dissection ouverte maximale pour $1,2,\ldots,i,n+2$ dont le premier et le dernier sommet ne sont utilisés que par un seul sous-polygone. On a ainsi $W_{1,1,i-1}U_{k-1,n-i+1}$ constructions possibles. Cela donne $\sum_{i=2}^{n-k+2} W_{1,1,i-1}U_{k-1,n-i+1}=\sum_{j=1}^{n-k+1} W_{1,1,j}U_{k-1,n-j}$ possibilités. Comme $U_{k-1,n-j}=0$ pour $j \geq n-k+2$ et $W_{1,1,0}=0$, on a :
\[W_{1,k,n}=\sum_{j=0}^{n} W_{1,1,j}U_{k-1,n-j}.\]

\noindent Ainsi, par le corollaire \ref{27}, $W_{1,k}=W_{1,1}U_{k-1}=W_{1,1}(1-P^{-1})^{k-1}$.
\\
\\iv) Soient $n \geq 1$, $(a_{1},\ldots,a_{n+2}) \in (\mathbb{N}^{*})^{n+2}$ et $\epsilon \in \{-1,1\}$. On a 
\begin{eqnarray*}
M_{n+2}(a_{1},\ldots,a_{n+2})=\epsilon TSTS & \Longleftrightarrow & TSM_{n+2}(a_{1},\ldots,a_{n+2})=\epsilon (TS)^{3} \\
                                            & \Longleftrightarrow & M_{n+3}(a_{1},\ldots,a_{n+2},1)=-\epsilon Id. \\
\end{eqnarray*}

\noindent Par conséquent, l'ensemble des solutions de taille $n+2$ de $(E_{TSTS})$ est en bijection avec l'ensemble des solutions de taille $n+3$ de $(E_{Id})$ dont la dernière composante est égale à 1. Ainsi, $x_{n}=V_{1,n+1}$.
\\
\\v) Soient $n \geq 1$, $(a_{1},\ldots,a_{n+2}) \in (\mathbb{N}^{*})^{n+2}$ et $\epsilon \in \{-1,1\}$. On a 
\begin{eqnarray*}
M_{n+2}(a_{1},\ldots,a_{n+2})=\epsilon STST & \Longleftrightarrow & M_{n+1}(a_{2},\ldots,a_{n+2})T^{a_{1}}S(ST)=\epsilon (ST)^{3} \\
                                            & \Longleftrightarrow & -M_{n+1}(a_{2},\ldots,a_{n+2})T^{a_{1}+1}=-\epsilon Id \\
																						& \Longleftrightarrow & T^{0}SM_{n+1}(a_{2},\ldots,a_{n+2})T^{a_{1}+1}S=\epsilon S^{2} \\
																						& \Longleftrightarrow & M_{n+3}(a_{1}+1,a_{2},\ldots,a_{n+2},0)=-\epsilon Id \\
																						& \Longleftrightarrow & M_{n+3}(a_{2},\ldots,a_{n+2},0,a_{1}+1)=-\epsilon Id \\
																						& \Longleftrightarrow & M_{n+1}(a_{2},\ldots,a_{n+1},a_{n+2}+a_{1}+1)=\epsilon Id. \\
\end{eqnarray*}

\noindent En particulier, si $(a_{1},\ldots,a_{n+2})$ est une solution de $(E_{STST})$ alors $(a_{2},\ldots,a_{n+1},a_{n+2}+a_{1}+1)$ est une solution de $(E_{Id})$ dont la dernière composante est supérieure à 3. Pour $n \leq 3$, il n'y a pas de $\lambda$-quiddité de taille $n+1$ dont la dernière composante est supérieure à 3. Donc, $y_{1}=y_{2}=y_{3}=0$.
\\
\\Notons, pour $n \geq 4$, $3 \leq d \leq n-1$ et $1 \leq k \leq d-2$, $\tilde{X}_{n}$ l'ensemble des solutions de $(E_{STST})$ de taille $n+2$, $\tilde{Y}_{n,d}:=\{(a_{1},\ldots,a_{n+2}) \in \tilde{X}_{n},~a_{1}+a_{n+2}=d-1\}$ et $\tilde{Z}_{n,d,k}$ le sous-ensemble de $\tilde{Y}_{n,d}$ suivant : $\{(a_{1},\ldots,a_{n+2}) \in \tilde{Y}_{n,d},~a_{1}=k~{\rm et}~a_{n+2}=d-k-1\}$. En reprenant les éléments développés dans la preuve du théorème \ref{12}, on a $\tilde{X}_{n}= \bigsqcup_{d=3}^{n-1} \tilde{Y}_{n,d}=\bigsqcup_{d=3}^{n-1} \bigsqcup_{k=1}^{d-2 }\tilde{Z}_{n,d,k}$. De plus, $\tilde{Z}_{n,d,k}$ est en bijection avec l'ensemble des $\lambda$-quiddités de taille $n+1$ dont la dernière composante est $d$. Par conséquent, ${\rm card}(\tilde{Z}_{n,d,k})=V_{d,n-1}$ et 
\[y_{n}=\sum_{d=3}^{n-1} \sum_{k=1}^{d-2} V_{d,n-1}=\sum_{d=3}^{n-1} (d-2) V_{d,n-1}.\]

\end{proof}

Bien que l'on n'ait pas eu besoin de toutes les valeurs de $W_{k,l}$ pour effectuer les dénombrements souhaités, on va conclure ce texte en démontrant une formule générique permettant d'obtenir toutes les valeurs des coefficients de cette série génératrice.

\begin{proposition}
\label{35}

Soient $n,k,l \geq 1$ et $W_{k,l,n}$ le nombre de $\lambda$-quiddités de taille $n+2$ dont la première composante vaut $k$ et la dernière composante est égale à $l$. On pose $W_{k,l,0}:=0$ et $W_{k,l}:=\sum_{n=0}^{+\infty} W_{k,l,n} X^{n}$. On a :
\[W_{k,l}=W_{1,1}(1-P^{-1})^{k-l-2}=W_{1,k+l-1}.\]

\end{proposition}

\begin{proof}

Soient $n,k,l \geq 1$. Si $k=1$ ou $l=1$ alors le résultat est vrai (voir la preuve du théorème \ref{34} iii)). On suppose maintenant $k,l \geq 2$ et on considère $\mathcal{P}$ un polygone convexe à $n+2$ sommets, ces derniers étant numérotés de 1 à $n+2$ dans le sens trigonométrique. Par le théorème \ref{23}, $W_{k,l,n}$ est le nombre de $3d$-dissections ouvertes maximales de $\mathcal{P}$ pour lesquels le sommet 1 est utilisé par $k$ sous-polygones et le sommet $n+2$ est utilisé par $l$ sous-polygones.
\\
\\Pour avoir une $3d$-dissection ouverte maximale de $\mathcal{P}$ dont le premier sommet est utilisé par $k$ sous-polygones et dont le dernier sommet est utilisé par $l$ sous-polygones, on commence par choisir le sommet du sous-polygone de base dont le numéro $k+1 \leq i \leq n+1$ est le plus petit, ce qui permet de construire le côté $1,i$. Ensuite, on choisit le sommet du sous-polygone de base dont le numéro $i \leq j \leq n-l+2$ est le plus grand, ce qui permet de construire le côté $j,n+2$ (si $k+1>n-l+2$ alors $W_{k,l,n}=0$). Pour terminer, on doit construire une $3d$-dissection ouverte maximale à base ouverte pour le polygone $1,\ldots,i$ dont le premier sommet est utilisé par $k-1$ sous-polygones, une $3d$-dissection ouverte maximale pour le polygone $1,i,\ldots,j,n+2$ dont le premier et le dernier sommet ne sont utilisés que par un seul sous-polygone et une $3d$-dissection ouverte maximale à base ouverte pour le polygone $j,\ldots,n+2$ dont le dernier sommet est utilisé par $l-1$ sous-polygones. Cela donne $U_{k-1,i-2}W_{1,1,j-i+1}U_{l-1,n-j+1}$ constructions possibles (car il y a autant de $3d$-dissections ouvertes maximales à base ouverte d'un polygone à $i$ sommets dont le premier sommet est utilisé par $k-1$ sous-polygones que de $3d$-dissections ouvertes maximales à base ouverte d'un polygone à $i$ sommets dont le dernier sommet est utilisé par $k-1$ sous-polygones). Ainsi, $W_{k,l,n}=\sum_{i=k+1}^{n+1} \sum_{j=i}^{n-l+2} U_{k-1,i-2}W_{1,1,j-i+1}U_{l-1,n-j+1}=\sum_{i=k-1}^{n-1} \sum_{j=i}^{n-l} U_{k-1,i}W_{1,1,j-i+1}U_{l-1,n-j-1}$. Comme $U_{k-1,h}=0$ pour $0 \leq h \leq k-2$, $U_{l-1,n-h-1}=0$ pour $h \geq n-l+1$ et $W_{1,1,0}=0$, on a :
\[W_{k,l,n}=\sum_{i=0}^{n-1} \sum_{j=i-1}^{n-1} U_{k-1,i}W_{1,1,j-i+1}U_{l-1,n-j-1}=\sum_{i=0}^{n} U_{k-1,i}\left(\sum_{s=0}^{n-i} W_{1,1,s}U_{l-1,(n-i)-s}\right).\]
 \noindent De plus, $W_{k,l,0}=U_{k-1,0}W_{1,1,0}U_{l-1,0}$. Par conséquent, on a :
\[W_{k,l}=U_{k-1}W_{1,1}U_{l-1}.\]
\noindent Par le corollaire \ref{27}, on obtient :
\[W_{k,l}=(1-P^{-1})^{k-1}W_{1,1}(1-P^{-1})^{l-1}=W_{1,1}(1-P^{-1})^{k+l-2}=W_{1,k+l-1}.\]

\end{proof}

\begin{examples}
{\rm On a $W_{2,2,6}=17$, $W_{2,3,8}=114$ et $W_{5,3,11}=492$.
}
\end{examples}

\begin{remark}
{\rm L'existence de l'égalité $W_{k,l}=W_{1,k+l-1}$ laisse envisager la possibilité d'avoir des bijections intéressantes au sein de l'ensemble des $\lambda$-quiddités.
}
\end{remark}

\subsection{Applications numériques}

Le tableau ci-dessous donne les valeurs de $W_{1,k,n}$ pour les petites valeurs de $k$ et $n$.

\hfill\break

\begin{center}
\begin{tabular}{|c|c|c|c|c|c|c|c|c|c|c|c|c|c|}
\hline
  \multicolumn{1}{|c|}{\backslashbox{$k$}{\vrule width 0pt height 1.25em$n$}} & 0 & 1 & 2 & 3 & 4 & 5 & 6 &7 & 8 & 9 & 10 & 11 & 12 \\
	\hline
  1   & 0 & 1 & 0 & 0 & 1 & 3 & 9 & 29 & 99 & 348 & 1247 & 4539 & 16~740     \rule[-7pt]{0pt}{18pt} \\
	\hline
  2   & 0 & 0 & 1 & 1 & 2 & 7 & 22 & 71 & 239 & 830 & 2948 & 10~655 & 39~063     \rule[-7pt]{0pt}{18pt} \\
	\hline
  3   & 0 & 0 & 0 & 1 & 2 & 5 & 17 & 57 & 194 & 678 & 2420 & 8781 & 32~292     \rule[-7pt]{0pt}{18pt} \\
	\hline
  4   & 0 & 0 & 0 & 0 & 1 & 3 & 9 & 32 & 114 & 408 & 1481 & 5445 & 20~235     \rule[-7pt]{0pt}{18pt} \\
	\hline
  5   & 0 & 0 & 0 & 0 & 0 & 1 & 4 & 14 & 53 & 200 & 751 & 2831 & 10~730     \rule[-7pt]{0pt}{18pt} \\
\hline
  6   & 0 & 0 & 0 & 0 & 0 & 0 & 1 & 5 & 20 & 81 & 323 & 1270 & 4969     \rule[-7pt]{0pt}{18pt} \\
\hline
  7   & 0 & 0 & 0 & 0 & 0 & 0 & 0 & 1 & 6 & 27 & 117 & 492 & 2022     \rule[-7pt]{0pt}{18pt} \\
\hline
  8   & 0 & 0 & 0 & 0 & 0 & 0 & 0 & 0 & 1 & 7 & 35 & 162 & 717     \rule[-7pt]{0pt}{18pt} \\
\hline
  9   & 0 & 0 & 0 & 0 & 0 & 0 & 0 & 0 & 0 & 1 & 8 & 44 & 217     \rule[-7pt]{0pt}{18pt} \\
\hline
  10   & 0 & 0 & 0 & 0 & 0 & 0 & 0 & 0 & 0 & 0 & 1 & 9 & 54     \rule[-7pt]{0pt}{18pt} \\
\hline
  11   & 0 & 0 & 0 & 0 & 0 & 0 & 0 & 0 & 0 & 0 & 0 & 1 & 10     \rule[-7pt]{0pt}{18pt} \\
\hline
  12   & 0 & 0 & 0 & 0 & 0 & 0 & 0 & 0 & 0 & 0 & 0 & 0 & 1     \rule[-7pt]{0pt}{18pt} \\
\hline
	
\end{tabular}
\end{center}

\hfill\break

Notons qu'en sommant, pour un $n$ fixé, les valeurs de $W_{1,k,n}$ présentes dans la colonne indexée par $n$ on retrouve la valeur de $V_{1,n}$.

\hfill\break

\begin{center}
\begin{tabular}{|c|c|c|c|c|c|c|c|c|c|c|c|c|}
\hline
  n & 1 & 2 & 3 & 4 & 5 & 6 & 7 & 8 & 9 & 10 & 11 & 12 \\
	\hline
  $\mathcal{S}_{n}$ & 0 & 0 & 1 & 4 & 14 & 50 & 182 & 670 & 2489 & 9326 & 35~219 & 133~940      \rule[-7pt]{0pt}{18pt} \\
	\hline
	$\mathcal{T}_{n}$ & 1 & 2 & 5 & 16 & 53 & 180 & 627 & 2232 & 8084 & 29~690 & 110~310 & 413~870      \rule[-7pt]{0pt}{18pt} \\
	\hline
	$u_{n}$ & 0 & 1 & 3 & 9 & 30 & 104 & 368 & 1324 & 4834 & 17~870 & 66~755 & 251~601    \rule[-7pt]{0pt}{18pt} \\
	\hline
	$v_{n}$ & 0 & 1 & 3 & 9 & 29 & 99 & 348 & 1247 & 4539 & 16~740 & 62~420 & 234~924    \rule[-7pt]{0pt}{18pt} \\
	\hline
	$w_{n}$ & 0 & 0 & 1 & 4 & 14 & 51 & 188 & 697 & 2602 & 9786 & 37~065 & 141~291     \rule[-7pt]{0pt}{18pt} \\
	\hline
	$x_{n}$ & 1 & 2 & 6 & 19 & 62 & 209 & 726 & 2580 & 9331 & 34~229 & 127~050 & 476~290    \rule[-7pt]{0pt}{18pt} \\
	\hline
	$y_{n}$ & 0 & 0 & 0 & 1 & 5 & 20 & 78 & 302 & 1165 & 4492 & 17~349 & 67~185     \rule[-7pt]{0pt}{18pt} \\
	\hline
	\end{tabular}
\end{center}

\hfill\break

\end{document}